\documentclass[12pt]{amsart}
\usepackage{color}
\usepackage{amssymb}
\usepackage{comment}
\usepackage{pdfpages}
\usepackage{graphicx}
\usepackage{dcolumn}

\RequirePackage[numbers]{natbib}
\RequirePackage[colorlinks=true, pdfstartview=FitV, linkcolor=blue,
  citecolor=blue, urlcolor=blue]{hyperref}
\usepackage{paralist}

\usepackage{tikz} 
\usepackage{dcolumn}
\usetikzlibrary{arrows,positioning}
\tikzset{
    >=stealth',
    pil/.style={
           ->,
           thick,
           shorten <=2pt,
           shorten >=2pt,}
}

\newcommand{\cC}{{\mathcal{C}}}

\usepackage{graphicx}
\graphicspath{{./figures/}}
\usepackage{amsfonts}
\usepackage{amsmath}
\usepackage{amsthm}
\usepackage{amssymb}
\usepackage{amsbsy}
\usepackage{epsfig}
\usepackage{fullpage}
\usepackage{natbib, mathrsfs} 
\usepackage{verbatim}
\usepackage[latin1]{inputenc}
\usepackage{mhequ}
\usepackage{algorithm}
\usepackage{algorithmic}

\numberwithin{equation}{section}

\def \be{\begin{equs}}
\def \ee{\end{equs}}
\def \P{\mathbb{P}}
\def \E{\mathbb{E}}

\newcommand \TV{\mathrm{TV}}

\def \tmix{\tau_{\mathrm{mix}}}
\def \tmixR{\tau_{\mathrm{mix}}^{\mathrm{RW}}}

\def \TV{\mathrm{TV}}

\def \O{\mathrm{Occ}}

\def \LL {\Lambda(L,d)}

\def \ck{\mathcal{K}}

\def \Th{\Theta}
\def \hTh{\widehat{\Theta}}

\def \GSE{G_{\mathrm{SE}}}
\def \GMH{G_{\mathrm{MH}}}
\def \QSE{Q_{\mathrm{SE}}}
\def \QMH{Q_{\mathrm{MH}}}
\def \PSE{\pi_{\mathrm{SE}}}
\def \PMH{\pi_{\mathrm{MH}}}

\def \USE{U_{n,k}}
\def \UMH{U_{n,k}^{\mathrm{MH}}}
\def \DUSE{\mathcal{E}_{\mathrm{U,SE}}}
\def \DUMH{\mathcal{E}_{\mathrm{U,MH}}}

\def \LMH{L_{n,k}^{\mathrm{MH}}}

\newtheorem{theorem}{Theorem}[section]
\newtheorem{lemma}[theorem]{Lemma}

\newtheorem{corollary}[theorem]{Corollary}
\newtheorem{prop}[theorem]{Proposition}

\theoremstyle{plain}
\newtheorem{thm}{Theorem}
\newtheorem*{thm-non}{Theorem}

\theoremstyle{definition}
\newtheorem{defn}[theorem]{Definition}
\newtheorem{remark}[theorem]{Remark}

\begin{document}

\title[Mixing times for a Constrained Ising Process]{Mixing times for a Constrained Ising Process on the Two-Dimensional Torus at Low Density}


\author{Natesh S. Pillai$^{\ddag}$}
\thanks{$^{\ddag}$pillai@fas.harvard.edu, 
   Department of Statistics,
    Harvard University, 1 Oxford Street, Cambridge
    MA 02138, USA}

\author{Aaron Smith$^{\sharp}$}
\thanks{$^{\sharp}$smith.aaron.matthew@gmail.com, 
   Department of Mathematics and Statistics,
University of Ottawa, 585 King Edward Avenue, Ottawa
ON K1N 7N5, Canada}

\maketitle





\begin{abstract}
We study a kinetically constrained Ising process (KCIP) associated with a graph $G$ and density parameter $p$; this process is an interacting particle system with state space $\{ 0, 1 \}^{G}$, the location of the particles. The `constraint' in the name of the process refers to the rule that a vertex cannot change its state unless it has at least one neighbour in state `1'. The KCIP has been proposed by statistical physicists as a model for the glass transition. In this note, we study the mixing time of a KCIP on the 2-dimensional torus $G = \mathbb{Z}_{L}^{2}$ in the low-density regime $p = \frac{c}{L^{2}}$ for arbitrary $0 < c < \infty$, extending our previous results for the analogous process on the torus $\mathbb{Z}_{L}^{d}$ in dimension $d \geq 3$. Our general approach is similar, but the extension requires more delicate bounds on the behaviour of the process at intermediate densities.
\end{abstract}



\section{Introduction} \label{SecProbDesc}

The kinetically constrained Ising process (KCIP) refers to a class of interacting particle systems introduced by physicists in \cite{FrAn84, FrAn85} to study the glass transition. These processes have also appeared outside of the computer science literature (see the surveys \cite{CFM14b,ChMa13} for examples).  In this paper, we analyze one of the simplest and most-studied processes introduced in \cite{FrAn84,FrAn85}, called the FA1f process. The FA1f process takes as parameters the underlying graph $G$ and the typical density $0 < p < 1$ of 1's at equilibrium. The mixing time $\tmix$ of this process at small density $p = \frac{c}{|G|}$ for fixed $0 < c < \infty$ is the subject of a well-known conjecture of Aldous \cite{AldList}:
\be 
\tmix \approx p^{-2}  \, \tmix^{(G)},
\ee  
where $\tmix^{(G)}$ is the mixing time of simple random walk on $G$. The conjecture is based on the heuristic that, near equilibrium, the FA1f process at low temperature behaves much like a simple random walk on $G$ with roughly $p \, |G|$ walkers, slowed down by a factor of $p^{-3}$.

In previous work \cite{pillai2015mixing}, we studied Aldous' conjecture in the case that the underlying graph is the torus $\mathbb{Z}_{L}^{d}$ in dimension $d \geq 3$. In that paper, we showed that Aldous' conjecture does not quite hold for these examples: while the heuristic is correct near equilibrium, the mixing time is governed by the much larger time it takes for the initial all-1's configuration to drift towards a more typical configuration with roughly $c$ 1's. As we show in \cite{pillai2015mixing} in the special case of the torus in dimension $d \geq 3$, this drift time can be related to the time it takes coalescing random walks on the same underlying graph to coalesce. In this paper, we extend our previous work to the more difficult case of $d=2$. 

We recall the definition of the FA1f process on a general connected finite graph $G = (V,E)$ with density parameter $0 < p < 1$.  For a set $S$, we denote by $\mathrm{Unif}(S)$ the uniform distribution on $S$. Define a reversible Markov chain $\{ X_t \}_{t \in \mathbb{N}}$ on the set of $\{0,1\}$-labellings of $G$ by the following update procedure. At each time $t \in \mathbb{N}$, choose 
\be [EqCiRep]
v_{t} &\sim \mathrm{Unif}(V), \\
p_{t} &\sim \mathrm{Unif}([0,1]).
\ee
If there exists $u \in V$ such that $(u,v_{t}) \in E$ and $X_{t}[u] = 1$, set $X_{t+1}[v_{t}] = 1$ if $p_{t} \leq p$ and set $X_{t+1}[v_{t}] = 0$ if $p_{t} > p$. If no such $u \in V$ exists, set $X_{t+1}[v_t] = X_{t}[v_t]$. In either case, set $X_{t+1}[w] = X_{t}[w]$ for all $w \in V \backslash \{ v_t \}$. \par 
Set $|V| = n$; for general points $x \in \{0,1,2,\ldots\}^{G}$, define $\vert x \vert = \sum_{v \in G} \textbf{1}_{x[v] \neq 0}$. Let $\pi$ denote the stationary distribution of $\{ X_t \}_{t \in \mathbb{N}}$. For $y \in \Omega$, this stationary distribution is given by
\be \label{eqn:pistat}
\pi(y) = {1 \over \mathcal{Z}_{\mathrm{KCIP}}}\, p^{|y|}(1-p)^{n-|y|}\, \textbf{1}_{|y| > 0},
\ee
where $\mathcal{Z}_{\mathrm{KCIP}} = 1 - (1-p)^n$ is the normalizing constant. Thus $\pi(y)$ is proportional to the Binomial$(n,p)$ distribution on the number of non-zero labels in $y \in \Omega$, conditional on having at least one non-zero entry. \par

We give some standard notation. Denote by $\mathcal{L}(X)$ the distribution of a random variable $X$. Recall that for distributions $\mu, \nu$ on a common measure space $(\Theta, \mathcal{A})$, the \textit{total variation} distance between $\mu$ and $\nu$ is given by 
\be 
\| \mu - \nu \|_{\TV} = \sup_{A \in \mathcal{A}} (\mu(A) - \nu(A)).
\ee 
The  \textit{mixing profile} for the KCIP Markov chain $\{ X_{t} \}_{t \in \mathbb{N}}$ on $\Omega$ with stationary distribution $\pi$ is given by 
\be 
\tau(\epsilon) = \inf \Big \{ t > 0 \, : \, \sup_{X_{0}  = x \in \Omega}\| \mathcal{L}(X_{t}) - \pi \|_{\TV} < \epsilon \Big \}
\ee 
for $0 < \epsilon < 1$. As usual, the \textit{mixing time} is defined as $\tmix = \tau \big( \frac{1}{4} \big)$.

For a positive integer $L \in \mathbb{N}$, let $\Lambda(L,d)$ denote the $d$-dimensional torus with $n = L^{d}$ points; this is a Cayley graph with vertex set, generating set and edge set given by 
\be 
V &= \mathbb{Z}_{L}^{d}, \\
\mathrm{Gen} &= \{ (1,0,0,\ldots,0), (0,1,0,\ldots,0), \ldots, (0,0,0,\ldots,1) \}, \\
E &= \Big \{ (u,v) \in V \times V  \, : u - v \in \pm \mathrm{Gen}  \}.
\ee 
Set
\be
 n = |\Lambda(L,d)|  = L^d.
\ee 
In this paper, we study the KCIP on a sequence of graphs $\{ \Lambda(L,d) \}_{L \in \mathbb{N}}$ with density 
\be \label{eqn:pcn}
p = p_{n} = {c \over n}
\ee
for some fixed  constant $0 < c < \infty $ and fixed dimension $d = 2$. \par

The following is our main result: 

\begin{thm} [Mixing of the Constrained Ising Process on the Torus] \label{ThmMainResult}
Fix $0 < c < \infty$ and $d = 2$; let $p=p_n$ be as in \eqref{eqn:pcn}. Then the mixing time of the $\mathrm{KCIP}$ on $\LL$ satisfies
\be 
C_{1} n^{3}  \leq \tmix \leq C_{2} n^{3} \log(n)^{14} 
\ee 
for some constants $C_{1}, C_{2}$ that may depend on $c$ but are independent of $n$. 
\end{thm}

\begin{remark}
We show in \cite{pillai2015mixing} that the mixing time in dimension $d \geq 3$ satisfies $n^{3} \lessapprox \tmix \lessapprox n^{3} \log(n)$. We conjecture that $\tmix \approx n^{3}$ for $d \geq 3$ and $\tmix \approx n^{3} \log(n)$ for $d=2$.
\end{remark}

For comparison, the mixing time of the simple random walk on $G = \LL$ is known to be $\tmixR \approx n^{\frac{2}{d}}$ (see, \textit{e.g.}, Theorem 5.5 of \cite{LPW09}), while the worst-case expected hitting time of 0 is given by $\tau_{\mathrm{hit}} \approx n$ when $d \geq 3$ and $\tau_{\mathrm{hit}} \approx n \log(n)$ when $d =2$ (see, \textit{e.g.}, Theorem 4 of \cite{Cox89}).  \par

In the statement of Theorem \ref{ThmMainResult} and throughout the paper, we assume that  the quantity $0 < c < \infty$ is fixed; only $n = L^{2}$ grows. In particular, in Theorem \ref{ThmMainResult} and all other calculations, bounds that are `uniform' are implied to be uniform only in $n$ and other explicitly mentioned variables; they will generally not be uniform in $c$. Throughout the paper, we will denote by $C$ a generic constant, whose value may change from one occurrence to the next, but is independent of $n$. 

The main difficulty in extending the results of \cite{pillai2015mixing} to the case $d=2$ stems from the fact that the mixing time of simple random walk on the torus is very small compared to the size of the torus in dimensions $d \geq 3$, while this is no longer the case in dimension $d=2$. As a consequence of this fact, the behaviour of the FA1f diverges substantially from the behaviour of coalescing random walks long before all the walkers have coalesced (see \cite{Cox89, Oliv12b}). Thus, in dimension $d=2$, we can no longer rely on comparing the KCIP directly to the coalescing process until the number of particles is close to equilibrium, which was the main technique of \cite{pillai2015mixing}. Instead, we now need to analyze the behaviour of the process when it is moderately far from equilibrium. Although we focus on the special case of the torus in these papers, we believe that these behaviours are typical of the KCIP on rapidly-mixing and slowly-mixing graphs respectively.

\subsection{Related Work}

KCIP models have attracted a great deal of interest recently, including applications to combinatorics, computer science, and other areas. The recent survey \cite{GST11}  discusses KCIPs throughout physics, while \cite{ChMa13,CFM14b} have useful surveys of places that the KCIP has appeared outside of the physics literature. Recent mathematical progress has included new bounds on the mixing properties of the KCIP in various regimes \cite{KoLa06, CMRT09, MaTo13, ChMa13, CFM14, CFM14b, CFM15}, and the very recent work \cite{martinelli2016towards} makes substantial progress towards a ``universal" approach for bounding relaxation times of kinetically-constrained processes.

\subsection{General Notation}

We recall some standard notation that will be used throughout the paper. For sequences $x = x(n),y = y(n)$ indexed by $\mathbb{N}$, we write $y = O(x)$ for $\limsup_{n \rightarrow \infty} \frac {|y(n)|}{|x(n)|} \leq C < \infty$ and $y = o(x)$ for $\limsup_{n \rightarrow \infty} \frac {|y(n)|}{|x(n)|} = 0$. We write $y = \Theta(x)$ if both $ y = O(x)$ and $x = O(y)$. Finally, we also write $y \lessapprox x$ or $x \gtrapprox y$ for $y = O(x)$, and $y \approx x$ for $y = \Theta(x)$, during calculations.

\section{A Roadmap for the Proof} \label{SecRoadmap}

Our proof strategy builds on and improves the approach \cite{pillai2015mixing}. We recall some notation from that paper, give a sketch of our proof of Theorem \ref{ThmMainResult}, and then explain where our refinements occur.

First, we note that there is an obvious bijection between the points of $\Omega = \{ 0,1\}^{G}$ and the sets $\tilde{\Omega} = \{ S \subset G\}$: if $X \in \tilde{\Omega}$, then $\textbf{1}_{X} \in \Omega$. We often use this bijection without explicit discusssion if there is no possibility of confusion. For example, if $X, Y \in \Omega$, we would write $X \cap Y$ as shorthand for $\text{1}_{\{u \, : \, X[u] = Y[u] = 1\}}$ or $|X|$ as shorthand for $\sum_{u} X[u]$. 

For $1 \leq k \leq {n\over 2}$, let $\Omega_{k} \subset \Omega$ be configurations of $k$ particles for which no two particles are adjacent, \textit{i.e.},
\be \label{EqDefOmegaK}
\Omega_k = \big \{ X \in \{0,1\}^G \, : \, \sum_{v \in V} X[v] = k, \sum_{(u,v) \in E} X[u]X[v] = 0 \big \}.
\ee 
Also set $\Omega' = \Omega \backslash \cup_{k=1}^{{n\over 2}} \Omega_{k}$. For each $k \leq \frac{n}{2}$, we will denote by $\tmix^{(k)}$ the mixing time of the trace of $X_t$ on $\Omega_k$ (See Definition \ref{DefTrace} of Section \ref{SecCompToExc} for the precise definition of the trace of a Markov chain). We denote by $\tmix^{(\leq k)}$ the mixing time of the trace of $X_t$ on $\cup_{i \leq k} \Omega_{i}$. Define the quantity 
\be
\O_k(\epsilon, N) = \sup_{x \in \Omega} \, \inf\Big \{T \geq 1 \, : \, X_{1} = x, \, \P\big(\sum_{s=1}^T 1_{X_s \in \cup_{i \leq k} \Omega_i} > N\big) > 1-\epsilon \Big\}.
\ee
For a fixed $N$ and small $\epsilon$, $\O_k (\epsilon, N)$ denotes the first time at which the occupation measure of $X_t$ in $ \cup_{i \leq k} \Omega_i$ exceeds $N$ with  probability at least $(1-\epsilon)$.

\noindent Our proof strategy for the upper bound in Theorem \ref{ThmMainResult} entails the following steps: 
\begin{itemize}
\item[{\bf Step 1.}] We show that for a universal constant $r = r(c)$  depending only on the constant $c$ from \eqref{eqn:pcn}, and slowly-growing sequence $k_\mathrm{max} = k_{\mathrm{max}}(c,n) \equiv r(c) \log(n)$, 
\be
\tmix = O\Big( \tmix^{( \leq k_{\max})} +  \O_{k_{\max}} \big(\frac{1}{8 k_\mathrm{max}},C \tmix^{( \leq k_{\max})} \big) \Big).
\ee
This is an immediate consequence of Lemma 2.1 of \cite{PiSm15}.

\item[{\bf Step 2.}] By a comparison argument using the simple exclusion process, we show that 
\be \label{IneqDesiredRestrictionMixingBound}
\tmix^{(k)} = O(n^{3} \log(n)^{3} )
\ee uniformly in $1 \leq k \leq k_{\max}$. See Lemma \ref{CorMixingTimeRestWalk}. 

\item[{\bf Step 3.}] By coupling the KCIP to a `colored' version of the coalesence process over short time periods, we show that the process
\be \label{EqDefNumVerts}
V_t = \sum_{v \in V} X_t[v]
\ee 
satisfies the `drift condition'
\be \label{IneqDesiredDriftBound}
\E[V_{t + \epsilon S(n)} - V_t | X_{t}] \leq -\delta V_t + C(n)
\ee
for some characteristic time scale $S(n) \approx n^{3}$ and bias size $C(n) \approx \log(n)$, and for fixed $\epsilon, \delta > 0$ independent of $n$. See 
Theorem \ref{LemmaContractionEstimate}.  

\item[{\bf Step 4.}] By another comparison argument, we show that
\be 
\tau_{\mathrm{mix}}^{(\leq k_{\max})} = O( \max_{1 \leq k \leq k_{\max}} \tau_{\mathrm{mix}}^{(k)} \, \log(n)^{13}).
\ee 
See Lemma \ref{LemmaMixingModDensity}.
\item[{\bf Step 5.}] Conclude from ${\bf Step\, 3}$ and ${\bf Step\, 4}$ that $\O_k \big(\frac{1}{8 k_\mathrm{max}},C \tmix^{( \leq k_{\max})} \big) = O(n^3 \log(n)^{13})$. See Proposition \ref{LemmaOccMeasureBound}. 
\end{itemize}
The result then follows immediately by combining the bounds in Steps {\bf 1, 4} and {\bf 5}.\par

The key difference between this paper and the approach in \cite{pillai2015mixing} occurs at {\bf Step 3}. In \cite{pillai2015mixing}, Inequality \eqref{IneqDesiredDriftBound} was proved directly when $d \geq 3$ with $S(n) = n^{3}$ and $C(n) = C < \infty$ constant. The analogous bound is false in dimension $d=2$ for small $\epsilon > 0$, and we instead show that it holds for $S(n) = n^{3}$ and $C(n) = \log(n)$ when $d = 2$. This change means that we require stronger bounds in several of the remaining stages of the proof. The version of Inequality \eqref{IneqDesiredDriftBound} in this paper establishes that $V_{t} \lessapprox \log(n)$ with large probability after an initial burn-in period of length $T \lessapprox n^{3} \log(n)$. This is much weaker than the bound $V_{t} \lessapprox 1$ obtained in \cite{pillai2015mixing}, and so we now need the comparison bounds in {\bf Step 2} and {\bf Step 4} above to hold up to $k \approx \log(n)$, rather than up to $k \approx 1$.

\section{Mixing at Very High Density: Drift Condition for $V_t$} \label{SecDriftCond}

Recall the process $V_{t} = \sum_{v \in \LL} X_{t}[v]$ from Equation \eqref{EqDefNumVerts}. In this section, we show roughly that $V_{t} = O(\log(n))$ with high probability for any $t \gg n^{3} \log(n)$. The proof of this fact follows almost immediately from our proof of the analogous fact in our previous paper \cite{pillai2015mixing}, and so we state only the small adjustments that are required.

Define $G_{t} = (V(G_{t}), E(G_{t}))$ to be the induced subgraph of $\LL$ with vertices $V(G_{t}) = \{ u \in \LL \, : \, X_{t}[u] = 1\}$, and define 
\be \label{eqn:yt}
\mathrm{ConnComp}(G_t) = \text{The number of connected components of }\, G_{t}.
\ee 
Let $\mathcal{F}_{t}$ denote the $\sigma$-algebra generated by the random variables $\{X_{s} \}_{s \leq t}$. The key result in this section is a drift condition on $\{ V_{t} \}_{t \in \mathbb{N}}$, which follows almost immediately from bounds in \cite{pillai2015mixing}:

\begin{theorem} \label{LemmaContractionEstimate}

There exists some constant $0 < \epsilon_{0} = \epsilon_{0}(c)$ independent of $n$ so that for all $0 < \epsilon < \epsilon_{0}$, there exist constants $C_{G} = C_{G}(\epsilon,c) < \infty$, $\alpha = \alpha(\epsilon,c) > 0$ and $N = N(\epsilon,c)$ so that, for all $n > N$,
\be \label{eqn:conddrift}
\E[V_{\epsilon n^{3}} | V_{1}] \leq (1-\alpha) V_{1} + C_{G} \log(n).
\ee 
\end{theorem} 

Before giving the proof, we recall the definition of the coalescent process on a finite graph (\cite{ClSu73, HoLi75}):

\begin{defn}[Coalescent Process] \label{defCoalProcSimp}
Fix a regular graph $G = (V,E)$ and parameters $k \in \mathbb{N}$, $q \in[0, \frac{1}{k}]$. A \textit{coalescent process on graph $G$ with $k$ initial particles and moving rate $q$} is a Markov chain $\{ Z_{s} \}_{s \in \mathbb{N}}$ on $G^{k}$. Let $O_{s} = \{ v \in G \, : \, \exists \, \, 1 \leq i \leq k \,\, \text{such that} \, Z_{s}[i] = v\}$ be the \textit{occupied sites} of $Z_{s}$. To evolve $Z_{s}$, we first choose $u_{s} \sim \mathrm{Unif}([0,1])$, $v_{s} \sim \mathrm{Unif}([O_{s}])$ and $u_{s} \sim \mathrm{Unif}( \{v \in G \backslash \{v_{s}\} \, : \, (v,v_{s}) \in E \})$ and set $I_{s} = \{i \, : \, Z_{s}[i] = v_{s} \}$. If $u_{s} \leq q | O_{s} |$, then set $Z_{s+1}[j] = u_{s}$ for all $j \in I_{s}$ and set $Z_{s+1}[j] = Z_{s}[j]$ for all $j \notin I_{s}$; otherwise, set $Z_{s+1}[j] = Z_{s}[j]$ for all $j$. 
\end{defn}

\begin{proof}[Proof of Theorem \ref{LemmaContractionEstimate}]
Let $\{Z_{t}\}_{t \in \mathbb{N}}$ be a coalescent process on $\LL$ with $V_{1}$ initial particles. Let $L_{t} = | O_{t} |$ be the number of occupied sites of $Z_{t}$, so that $L_{1} = V_{1}$. Inequality (4.1) of \cite{Cox89} states that there exists a constant $0 < C < \infty$ so that, for all $t \in \mathbb{N}$,
\be 
\E[L_{t}] \leq C n \frac{\log(t)}{t-1}
\ee   
uniformly in the number $L_{1} = V_{1}$ of initial particles. In particular, we have
\be \label{IneqCoxNumCollisions}
\E[L_{\epsilon n}] \leq C \epsilon \log(n).
\ee 
Define the \textit{number of collisions by time $s$} to be 
\be \label{EqDefCCs}
\cC_{s} = \vert \{ 1 \leq u < s \, : \, \mathrm{ConnComp}(G_{u+1}) < \mathrm{ConnComp}(G_{u}) \} |.
\ee 
We obtain a lower bound on the number of collisions by following exactly the argument given for a similar bound in Lemma 6.15 of \cite{pillai2015mixing}, making and propagating two minor changes: 
\begin{enumerate}
\item We replace Inequality (6.47) of \cite{pillai2015mixing} and all references to the associated Theorem 5 of \cite{Cox89} with our Inequality \eqref{IneqCoxNumCollisions} and references to Inequality (4.1) of \cite{Cox89}.\footnote{Because of the different notation, Inequality (6.47) of \cite{pillai2015mixing} looks slightly different from our Inequality \eqref{IneqCoxNumCollisions} at first glance. In the notation of \cite{pillai2015mixing}, our Inequality \eqref{IneqCoxNumCollisions} would be written as $\E[\sum_{i} \textbf{1}_{\mathcal{A}_{2}^{(i)}}] \leq C \log(n)$. }
\item We replace the universal constant $C$ first defined in Inequality (6.47) of \cite{pillai2015mixing} with $C \log(n)$.
\end{enumerate}
For any fixed $\epsilon > 0$, the resulting lower bound on the number of collisions is
\be \label{IneqLowerBoundNumberOfConclusions}
\E[\cC_{\epsilon n^{3}}] \geq \alpha V_{1} - C \log(n)
\ee 
for some constants $0 < \alpha < 1$, $0 \leq C < \infty$ that may depend on $c$ and $\epsilon$, but which do not depend on $n$.

Inequality \eqref{eqn:conddrift} follows by an argument identical to the proof of Theorem 6.1 of \cite{pillai2015mixing}, with one change: we replace all references to Lemma 6.15 of \cite{pillai2015mixing} with references to our 
Inequality \eqref{IneqLowerBoundNumberOfConclusions}. The proof of Theorem 6.1 in \cite{pillai2015mixing} is fairly long, so we include a basic sketch of the argument here. The main idea is to couple the KCIP to a simple exclusion process in such a way that a positive percentage of collisions in the simple exclusion process occur shortly before a connected component of the KCIP is removed; this allows us to connect the bound in \eqref{IneqLowerBoundNumberOfConclusions} to the behaviour of the KCIP. The proof itself is concerned with checking that the coupling is tight enough for this transfer of information, and also checking that only a moderate number of new particles can be spawned by the KCIP over the relevant time interval.
\end{proof}

\section{Mixing at Moderate Densities: Trace of KCIP on $\Omega_k$} \label{SecCompToExc}

In this section, we bound the mixing time of the \textit{trace} of $\{X_{t}\}_{t \in \mathbb{N}}$ onto the sets $\Omega_{k}$ defined in Equation \eqref{EqDefOmegaK}, for all $k = O(\log(n))$. We recall the definition of the \textit{trace} of a Markov chain:

\begin{defn} [Trace] \label{DefTrace}
Fix an irreducible Markov chain $\{ Z_{t} \}_{t \in \mathbb{N}}$ on a finite state space $\Theta$. For a fixed subset $S \subset \Theta$, set $\eta(0) = 0$ and for $s \in \mathbb{N}$, recursively define the sequences of times 
\be \label{EqDefOccupationMeasureCounters}
\eta(s) &= \inf \{ t > \eta(s-1) \, : \, Z_{t} \in S \}, \\
\kappa(s) &= \sup \{u \, : \, \eta(u) \leq s \}.
\ee 
The quantity $\kappa$ can also be written as
\be \label{eqn:kappak}
\kappa(T) = \sum_{t=1}^{T} \textbf{1}_{Z_{t} \in S}.
\ee 
Then the \textit{trace} $\{ Z_{t}^{(S)}\}_{t \in \mathbb{N}}$ of the Markov chain $\{Z_{t}\}_{t \in \mathbb{N}}$ onto the set $S$ is given by
\be \label{EqDefRestChain}
Z_{t}^{(S)} = Z_{\eta(t)}. 
\ee 
\end{defn}

Fix $1 \leq k \leq \frac{n}{2}$, and let $Q_{n,k}$ be the kernel of the trace of $\{X_{t}\}_{t \in \mathbb{N}}$ on $\Omega_{k}$. Denote by $\tau_{n,k}$ the mixing time of $Q_{n,k}$ and denote by $1-\beta_{1}(Q_{n,k})$ the spectral gap of $Q_{n,k}$ (see Equation \eqref{EqVarCharAlpha} below for a definition of spectral gap). The key result of this section is: 

\begin{lemma} [Mixing of Restricted Walks] \label{CorMixingTimeRestWalk}
Fix $r \geq 1$. With notation as above, there exists a constant $C = C(c,r)$  that does not depend on $n$ so that
\be 
\tau_{n,k} &\leq  C n^{3} \log(n)^{3} \\
\frac{1}{1 - \beta_{1}(Q_{n,k})} &\leq C n^{3} \log(n)^{2} 
\ee 
uniformly in $1 \leq k \leq r \log(n)$ for all $n > N(c,r)$ sufficiently large.
\end{lemma}

We will proceed by using comparison theory, a tool developed for comparing the mixing properties of a Markov chain of interest to those of a similar and better-understood chain (see, \textit{e.g.}, \cite{DiSa93b} or \cite{DGJM06} for an introduction to this method). We prove our estimates on $Q_{n,k}$ by comparing the log-Sobolev constants of a sequence of  other well-studied Markov chains. We outline this sequence of comparison bounds, with notation collected here for easy reference:

\begin{enumerate}
\item Following \cite{pillai2015mixing}, we will first compare $Q_{n,k}$ to a sped-up and restricted version of the simple exclusion process (SE) on  $\LL$, whose kernel is denoted $\QMH$; see Section \ref{SubsecSeKcip}. The papers \cite{ClSu73, HoLi75} give an introduction to the simple exclusion process. 
\item We will next compare the modified version of the SE process with kernel $\QMH$ to a suitably modified Bernoulli-Laplace diffusion process, whose kernel is denoted $\UMH$. The original comparison paper \cite{DiSa93b} of Diaconis and Saloff-Coste compares the usual SE process to the standard Bernoulli-Laplace diffusion process. We use an argument very similar to that started in Section 3 of \cite{DiSa93b} and completed in Section 4.6 of \cite{DiSa96c}; see Section \ref{SubsecCompDirichletLaplaceSE}.
\item We use direct computations and a simple argument from \cite{Smit14a} to estimate the log-Sobolev constant of our modified Bernoulli-Laplace diffusion process $\UMH$.  See Section \ref{SubsecCompDirichletLaplace}.
\end{enumerate}

We next recall the definitions of the simple exclusion process and the Bernoulli-Laplace diffusion process, which form the basis of our kernels $\QMH$ and $\UMH$: 

\begin{defn} [Simple Exclusion Process on $\LL$] \label{DefSimpleExclusion}
The simple exclusion process $\{ Z_{t} \}_{t \in \mathbb{N}}$ is a Markov chain on the finite state space
\be \label{EqDefOmegaTildeK}
\Omega^{\mathrm{SE}}_{n,k} \equiv \{ Z \in \{ 0, 1 \}^{n} \, : \, \sum_{i} Z[i] = k \}.
\ee
To update $Z_{t}$, choose two adjacent vertices $u_{t}, v_{t} \in \Lambda(L,d)$ uniformly at random and set
\be 
Z_{t+1}[u_{t}] &= Z_{t}[v_{t}], \\
Z_{t+1}[v_{t}] &= Z_{t}[u_{t}] 
\ee 
and $Z_{t+1}[w] = Z_{t}[w]$ for $w \notin \{ u_{t}, v_{t} \}$. We denote by $Q_{n,k}^{\mathrm{SE}}$ the associated transition kernel.
\end{defn}

\begin{defn} [Bernoulli-Laplace Diffusion Process] \label{DefBernLaplace}
The Bernoulli-Laplace diffusion process $\{Z_{t}\}_{t \in \mathbb{N}}$ is a Markov chain on the finite state space $\Omega^{\mathrm{SE}}_{n,k}$ given in Equation \eqref{EqDefOmegaTildeK}. To update $Z_{t}$, sample
\be 
u_{t} &\sim \mathrm{Unif}( \{ i \, : \, Z_{t}[i] = 1 \}) \\
v_{t} &\sim \mathrm{Unif}( \{ i \, : \, Z_{t}[i] = 0 \}) \\
\ee 
and set 
\be 
Z_{t+1}[u_{t}] &= 0, \\
Z_{t+1}[v_{t}] &= 1, \\
Z_{t+1}[w] &= Z_{t}[w], \qquad w \notin \{u_{t},v_{t}\}. \\
\ee 
We denote by $U_{n,k}'$ the associated transition kernel and let $\USE = \frac{1}{2} U_{n,k} + \frac{1}{2} \mathrm{Id}$.
\end{defn}

\subsection{Comparison of Markov chains using Dirichlet forms}
Before proving the main result of this section, we recall some relevant results for comparing Dirichlet forms.  

\begin{defn} [Norms, Forms and Related Functions]
For a general Markov chain on a finite state space $X$ with kernel $P$ and unique stationary distribution $\pi$, and any function $f \, : \, X \rightarrow \mathbb{R}$ that is not identically 0, we respectively define  the $L_2$ norm, variance, Dirichlet form and entropy as:
\be[FunctionalDefs]
\| f \|_{2, \pi}^{2} &= \sum_{x \in X} \vert f(x) \vert^{2} \pi(x),\\
V_{\pi}(f) &= \frac{1}{2} \sum_{x,y \in X} \vert f(x) - f(y) \vert^{2} \pi(x) \pi(y), \\
\mathcal{E}_{P}(f,f) &= \frac{1}{2} \sum_{x,y \in X} \vert f(x) - f(y) \vert^{2} P(x,y) \pi(x), \\
L_{\pi}(f) &= \sum_{x \in X} \vert f(x) \vert^{2} \log \big( \frac{f(x)^{2}}{\| f \|_{2,\pi}^{2}} \big) \pi(x). 
\ee
Recall that the \textit{log-Sobolev constant} and \textit{spectral gap} of a Markov transition matrix $P$ are given by
\be[EqVarCharAlpha]
\alpha(P) &= \inf_{f \neq 0} \frac{\mathcal{E}_{P}(f,f)}{L_{\pi}(f)} \\
1 - \beta_{1}(P) &=  \inf_{f \neq 0} \frac{\mathcal{E}_{P}(f,f)}{V_{\pi}(f)}.
\ee
\end{defn}

Fix two finite state spaces $\Th \subset \hTh$. Let $K,Q$ be the kernels of two $\frac{1}{2}$-lazy, aperiodic, irreducible, reversible Markov chains. Assume that $K$ has stationary measure $\mu$ on a state space $\hTh$ while $Q$ has stationary measure $\nu$ on a state space $\Th \subset \hTh$. Denote by $f$ a function on $\Th$, and call a function $\widehat{f}$ on $\hTh$ an \emph{extension} of $f$ if  $\widehat{f}(x) = f(x)$ for all $x \in \Th$.

Next, fix a family of probability measures $\{ \P_{x}[\cdot] \}_{x \in \hTh}$ on $\Th$ that satisfy $\P_{x}[\cdot] = \delta_{x}(\cdot)$ for $x \in \Th$. We will use only extensions of the form
\be \label{EqLinearFamilyExtensions}
\widehat{f}(x) = \sum_{y \in {\Th}} \P_{x}[y] f(y).
\ee
We call extensions of the form \eqref{EqLinearFamilyExtensions} \textit{linear extensions}. \par
Fix a linear extension. For each pair $(x,y) \in \hTh$ with $K(x,y) > 0$, fix a joint probability distribution $\P_{x,y}$ on $\Th \times \Th$ satisfying $\sum_{a} \P_{x,y}[a,b] = \P_{y}[b]$ for all $b \in \Th$ and $\sum_{b} \P_{x,y}[a,b] = \P_{x}[a]$ for all $a \in \Th$. This is a coupling of the distributions $\P_{x}, \P_{y}$. \par 

\begin{defn}[Paths, Flows] \label{DefFlowDistPath}
Finally, for each $a,b \in \Th$ with $\sum_{x,y \in \hTh} \P_{x,y}[a,b] > 0$, we define a \emph{flow} in $\Th$ from $a$ to $b$. To do so, call a sequence of vertices $\gamma = [ a = v_{0,a,b}, v_{1,a,b}, \ldots, v_{k[\gamma], a,b} = b ]$ a \emph{path} from $a$ to $b$ if $Q(v_{i,a,b}, v_{i+1,a,b}) > 0$ for all $0 \leq i < k[\gamma]$. Then let $\Gamma_{a,b}$ be the collection of all paths from $a$ to $b$ and let $\Gamma = \cup_{a,b} \Gamma_{a,b}$. Call a function $F \, : \, \Gamma \mapsto [0,1]$ a \emph{flow} if $\sum_{\gamma \in \Gamma_{a,b}} F[\gamma] = 1$ for all $a,b$. For a path $\gamma \in \Gamma_{a,b}$, we will label its initial and final vertices by $i(\gamma) = a$, $o(\gamma) = b$. 
\end{defn}

The purpose of these definitions is to provide a way to compare the functionals described in Equation \eqref{FunctionalDefs}. If there exists a family of measures $\{ \P_{x} \}_{x \in \hTh}$ so that the associated linear extensions given by formula \eqref{EqLinearFamilyExtensions} satisfy
\be
L_{\nu}(f) &\leq C_{L}\, L_{\mu}(\widehat{f}), \\
\mathcal{E}_{K}(\widehat{f}, \widehat{f}) &\leq C_{\mathcal{E}}\, \mathcal{E}_{Q}(f,f),
\ee
\noindent then the variational characterization of $\alpha$ given in formula \eqref{EqVarCharAlpha} implies
\be \label{IneqContentlessLogSobCompBound} 
\alpha(Q) &\geq \frac{1}{C_{L} C_{\mathcal{E}}} \alpha(K).
\ee

This is the motivation for Theorem 4 and Lemma 2 of \cite{Smit14a}. 
Theorem 4  of \cite{Smit14a} may be restated as:

\begin{thm}[Comparison of Dirichlet Forms for General Chains] \label{ThmDirGenChain}
Let $K,Q$ be the kernels of two  reversible Markov chains. Assume that $K$ has stationary measure $\mu$ on state space $\hTh$ while $Q$ has stationary measure $\nu$ on state space $\Th \subset \hTh$. Fix flow $F$, distributions $\P_{x}$ and couplings $\P_{x,y}$ as in the notation in Definition \ref{DefFlowDistPath} above. Then for any function $f$ on $\Th$ and the linear extension $\hat{f}$ of $f$ on $\hTh$ given by formula \eqref{EqLinearFamilyExtensions},
\begin{equation*}
\mathcal{E}_{K}(\widehat{f}, \widehat{f}) \leq \mathcal{A} \mathcal{E}_{Q}(f,f),
\end{equation*}
where 
\begin{align*}
\mathcal{A} = \sup_{Q(q,r) >0}  \frac{1}{Q(q,r) \nu(q)} & \Big( \sum_{\gamma \ni (q,r)} F[\gamma] k[\gamma] K(i(\gamma),o(\gamma)) \mu(i(\gamma)) \\
&\hspace{1cm}+ 2  \sum_{\gamma \ni (q,r)} k[\gamma] F[\gamma] \sum_{y \in \hTh \backslash \Th} \P_{y}[o(\gamma)] K(i(\gamma),y) \mu(i(\gamma)) \\
&\hspace{1cm}+ \sum_{\gamma \ni (q,r)}  k[\gamma] F[\gamma] \sum_{x,y \in \hTh \backslash \Th \, : \, K(x,y) > 0} \P_{x,y}[i(\gamma),o(\gamma)] K(x,y) \mu(x) \Big).
\end{align*}
\end{thm}

Lemma 2 of \cite{Smit14a} may be restated as:

\begin{lemma} [Comparison of Variance and Log-Sobolev Constants] \label{LemmaVarLogSobComp}
Let $\mu$ be a measure on $\hTh$ and $\nu$ be a measure on $\Th \subset \hTh$. Let $\tilde{C} = \sup_{y \in \Omega} \frac{\nu(y)}{\mu(y)}$. Then for any function $f$ on $\Th$ and linear extension $\hat{f}$ of $f$ on $\hTh$,
\be
V_{\nu}(f) &\leq \tilde{C} V_{\mu}(\widehat{f}), \\
L_{\nu}(f) &\leq \tilde{C} L_{\mu}(\widehat{f}). 
\ee
\end{lemma}

\subsection{The log-Sobolev Constant of a Modified Dirichlet-Laplace Diffusion Processes} \label{SubsecCompDirichletLaplace}

Let $\USE$ be as in Definition \ref{DefBernLaplace} and let $\UMH$ be the Metropolis-Hastings chain with proposal distribution $\USE$ and target distribution the uniform distribution on $\Omega_{n,k} \equiv \Omega_{k}$. We define $\PMH$ to be the uniform distribution on $\Omega_{n,k}$ and $\PSE$ to be the uniform distribution on $\Omega^{\mathrm{SE}}_{n,k}$. Let $\DUSE$ and $\DUMH$ be the Dirichlet forms associated with $\USE$ and $\UMH$. The main bound in this section is: 

\begin{lemma} \label{LemmaDirichletLaplaceMainCompBound}
Fix $0 < r < \infty$. Let $\alpha(\UMH)$ and $1- \beta_{1}(\UMH)$ be the log-Sobolev constant and spectral gap of $\UMH$. Then there exists some constant $C = C(c,r) < \infty$ that does not depend on $n$ so that 
\be 
\alpha(\UMH) &\geq \frac{C }{n \log(n)^{3} } \\
1 - \beta_{1}(\UMH) &\geq \frac{C}{n \log(n)^{2}}
\ee  
uniformly in $1 \leq k \leq r \log(n)$.
\end{lemma}

Before proving this, we recall an estimate of the log-Sobolev constant of the ``perfect" transition kernel $\LMH$ on $\Omega_{n,k}$, defined by 
\be
\LMH(x,y) = \frac{1}{2| \Omega_{n,k} |} + \frac{1}{2} \textbf{1}_{x=y}. \\
\ee 

We have: 

\begin{lemma} [Log-Sobolev Constant of $\LMH$] \label{LemmaLogSobSilly}
Fix $0 < r < \infty$. Let $\alpha(\LMH)$ and $1 - \beta_{1}(\LMH)$ be the log-Sobolev constant and spectral gap of $\LMH$. Then there exists some constant $0 < C = C(c,r) < \infty$ that does not depend on $n$ so that 
\be 
\alpha(\LMH) &\geq \frac{C }{\log(n)} \\
1 - \beta_{1}(\LMH) &\geq C
\ee  
uniformly in $1 \leq k \leq r \log(n)$, for all $n > N(r)$ sufficiently large.
\end{lemma}

\begin{proof}
This follows immediately from an application of Inequality (3.10) of \cite{DiSa96c} and the well-known fact that the spectral gap of $\LMH$ is $\Theta(1)$. 
\end{proof}

We are now ready to prove Lemma \ref{LemmaDirichletLaplaceMainCompBound} by comparing $\UMH$ to $\LMH$:

\begin{proof} [Proof of Lemma \ref{LemmaDirichletLaplaceMainCompBound}]

We will apply Theorem 2.1 of \cite{DiSa93b} (this result is equivalent to the special case of Theorem \ref{ThmDirGenChain} when $\Th = \hTh$, so we do not restate the result). Since $\Th = \hTh$, we do not need to define distributions or couplings; we need only define the relevant paths and flows on those paths. We assume that $n > 20k$. 

We define our random paths below. The intuition behind these paths is as follows. There is an obvious path between any pair $X,Y \in \Omega_{n,k}$: simply move particles in $X$ to particles in $Y$ one at a time, in any order. Unfortunately, for some choices of $X$, $Y$, this obvious path will leave the state space $\Omega_{n,k}$. To avoid this problem, we sample a random intermediate point $Z$ at random; with high probability, the direct paths from $X$ to $Z$ and from $Z$ to $Y$ will remain in $\Omega_{n,k}$ and the additional steps will not have a large impact on the final bound.

\begin{defn} [Flows for Bernoulli-Laplace Diffusions] \label{DefRandomFlowBernoulliLaplace}
Fix $X,Y \in \Omega_{n,k}$. We sample a length-2 path from $X$ to $Y$ by the following algorithm: 
\begin{enumerate}
\item Choose $Z$ uniformly from the set 
\be 
\Omega^{X,Y} = \{ Z \in \Omega_{n,k} \, : \, \sum_{|u-v| \leq 1} (X[u] + Y[u]) Z[v] = 0\}
\ee 
of configurations that have no particles next to either $X$ or $Y$. 
\item Let  
\be 
\{ x_{1},\ldots,x_{k}\} &= \{ u \, : \, X[u] =1 \} \\
\{ y_{1},\ldots,y_{k}\} &= \{ u \, : \, Y[u] =1 \} \\
\{ z_{1},\ldots,z_{k}\} &= \{ u \, : \, Z[u] =1 \} \\
\ee 
be the location of all particles in $X,Y$ and $Z$ respectively, ordered uniformly at random. 
\item Define a path $P_{1}^{X,Y} = (\sigma_{1},\ldots,\sigma_{k+1})$ from the set associated with $X$ to the set associated with $Z$ by
\be 
\sigma_{i}[j] &= x_{j}, \qquad i \leq j \\
\sigma_{i}[j] &= z_{j}, \qquad i > j. \\
\ee 
Define a path $P_{2}^{X,Y} = (\eta_{1},\ldots,\eta_{k+1})$ from the set associated with $Z$ to the set associated with $Y$ by
\be 
\eta_{i}[j] &= z_{j}, \qquad i \leq j \\
\eta_{i}[j] &= y_{j}, \qquad i > j. \\
\ee 
\item Return the path $P^{X,Y} = (\textbf{1}_{\sigma_{1}},\ldots, \textbf{1}_{\sigma_{k+1}}, \textbf{1}_{\eta_{2}},\ldots,\textbf{1}_{\eta_{k+1}})$ from $X$ to $Y$.
\end{enumerate}
\end{defn}

Having defined the flows, we have implicitly defined the constant $\mathcal{A}$ in Theorem  \ref{ThmDirGenChain}. We must now bound that constant. Fix a pair of elements $(Q,R)$ with $\UMH(Q,R) > 0$ and $Q \neq R$. By the definition of $\UMH$, we must have that $| Q \backslash R| = |R \backslash Q| = 1$. For $X, Y \in \Omega_{n,k}$, let $P^{X,Y}$ be a random path as given by Definition \ref{DefRandomFlowBernoulliLaplace} and let $F$ be the associated flow. In order to bound the weight assigned to the edge $(Q,R)$, we note that all paths have length at most $2k$, and so 
\be \label{IneqDiffRewrite}
\sum_{X,Y \in \Omega_{n,k}} \sum_{\gamma \in \Gamma_{X,Y} \, : \, (Q,R) \in \gamma} |\gamma| \, F[\gamma] &\leq  2k \sum_{X,Y \in \Omega_{n,k}} \sum_{\ell = 1}^{k} \Big(\P[(Q,R) = (\sigma_{\ell}, \sigma_{\ell+1})] \\
&\hspace{4cm}+ \P[(Q,R) = (\eta_{\ell}, \eta_{\ell+1})]\Big).
\ee

We note that $P_{1}^{X,Y}$ and $P_{2}^{X,Y}$ are symmetric. Thus, to bound the weight \eqref{IneqDiffRewrite} assigned to the edge $(Q,R)$, it is enough to bound $\P[(Q,R) = (\sigma_{\ell}, \sigma_{\ell+1})]$ for all fixed $1 \leq \ell \leq k$ and $X,Y$. To do so, we note that it is possible to sample from $P_{1}^{X,Y}$ using the following rejection-sampling algorithm:

\begin{enumerate}
\item Choose $\hat{Z}$ uniformly from the set $\{ z \in \{0,1\}^{G} \, : \, \sum_{v \in G} z[v] = k\}$.
\item Let  
\be 
\{ x_{1},\ldots,x_{k}\} &= \{ u \, : \, X[u] =1 \} \\
\{ y_{1},\ldots,y_{k}\} &= \{ u \, : \, Y[u] =1 \} \\
\{ \hat{z}_{1},\ldots,\hat{z}_{k}\} &= \{ u \, : \, Z'[u] =1 \} \\
\ee 
be the location of all particles in $X,Y$ and $Z'$ respectively, ordered uniformly at random. 
\item Define a path $P_{1}^{X,Y} = (\sigma_{1},\ldots,\sigma_{k+1})$ from the set associated with $X$ to the set associated with $Z'$ by
\be 
\sigma_{i}[j] &= x_{j}, \qquad i \leq j \\
\sigma_{i}[j] &= \hat{z}_{j}, \qquad i > j. \\
\ee 
Define the associated proposal path $\hat{\gamma} = 
(\textbf{1}_{\sigma_{1}},\ldots, \textbf{1}_{\sigma_{k+1}}, \textbf{1}_{\eta_{2}},\ldots,\textbf{1}_{\eta_{k+1}})$.
\item If $\hat{Z} \in \Omega^{X,Y}$, say that we \textit{accept} this path and return the path $\hat{\gamma}$. Otherwise, say that we \textit{reject} this choice of $\hat{Z}$ and go back to step 1 of this algorithm.
\end{enumerate}

Note that this algorithm makes sense even if $X,Y$ are not in $\Omega_{n,k}$. We note that, for $\hat{\gamma}$ as in step 3 of the algorithm, we can compute exactly 
\be 
\sum_{X,Y \subset \{0,1\}^{G} \, : \, |X| = |Y|=k} \P[(Q,R) = (\hat{\gamma}[\ell], \hat{\gamma}[\ell+1])] = {n \choose k-1}.
\ee 
Furthermore, for $X,Y \in \Omega_{n,k}$, it is clear that 
\be 
\P[\hat{Z} \, \text{ is rejected.}] = O\left(\frac{k}{n} \right)  = O\left(\frac{\log(n)}{n} \right) =  o(1). 
\ee 
Combining these two bounds, we have:
\be 
\sum_{X,Y \in \Omega_{n,k}}  \P[(Q,R) = (\sigma_{\ell}, \sigma_{\ell+1})] \leq   {n \choose k-1}(1 + o(1)).
\ee

Combining this with Inequality \eqref{IneqDiffRewrite}, we have
\be \label{IneqCompDiffSilly1}
\sum_{X,Y \in \Omega_{n,k}} \sum_{\gamma \in \Gamma_{X,Y} \, : \, (Q,R) \in \gamma} |\gamma| \, F[\gamma] \leq 2k \frac{n^{k-1}}{(k-1)!} (1 + o(1)).
\ee 
Note that $\UMH$ and $\LMH$ have the same stationary distribution, and that
\be \label{IneqCompDiffSilly2}
\frac{\UMH(x,y)}{\LMH(x,y)} = \frac{n^{k-1}}{(k-1)!}(1 + o(1))
\ee 
for any $(x,y)$ for which $\UMH(x,y) \neq 0$. Combining Inequalities \eqref{IneqCompDiffSilly1} and \eqref{IneqCompDiffSilly2}, we conclude that our choice of flow yields a value of $\mathcal{A}$ in Theorem \ref{ThmDirGenChain} that satisfies 
\be 
\mathcal{A} \leq 4 k^{2} (1 + o(1)).
\ee 
The results follow immediately from applying Theorem \ref{ThmDirGenChain} with this bound on $\mathcal{A}$ and the bound on the log-Sobolev constant (respectively spectral gap) of $\LMH$ obtained in Lemma \ref{LemmaLogSobSilly}.

\end{proof}

\subsection{Comparing Modified Dirichlet-Laplace Diffusion Process to Modified Simple Exclusion Process} \label{SubsecCompDirichletLaplaceSE}

For $n \in \mathbb{N}$ and $1 \leq k \leq \frac{n}{2}$, we define the graphs $\GSE = (V_{\mathrm{SE}}, E_{\mathrm{SE}})$ and $\GMH = (V_{\mathrm{MH}}, E_{\mathrm{MH}})$ by 
\be 
V_{\mathrm{SE}} &= \Omega_{n,k}^{\mathrm{SE}} \\
V_{\mathrm{MH}} &= \Omega_{n,k} \\
E_{\mathrm{SE}} &= \{(u,v) \in V_{\mathrm{SE}} \, : \,  Q_{n,k}^{\mathrm{SE}}(u,v) > 0 \} \\
E_{\mathrm{MH}} &= E_{\mathrm{SE}} \cap V_{\mathrm{MH}}^{2}, \\
\ee
where $\Omega_{n,k}^{\mathrm{SE}}$ and $Q_{n,k}^{\mathrm{SE}}$ are given in Definition \ref{DefSimpleExclusion}, and $\Omega_{n,k} = \Omega_{k}$ is defined in Equation \eqref{EqDefOmegaK}. Note that $\GMH$ is a subgraph of $\GSE$.

We then define $\QSE$ to be the kernel of the $\frac{1}{2}$-lazy simple random walk on $\GSE$; this has stationary distribution $\PSE$ that is uniform on $V_{\mathrm{SE}}$. We define $\PMH$ to be the uniform distribution on $V_{\mathrm{MH}}$ and define $\QMH$ to be the Metropolis-Hastings kernel with proposal kernel $\QSE$ and stationary measure $\PMH$. That is,
\be 
\QMH(x,y) = \QSE(x,y) \textbf{1}_{x,y \in V_{\mathrm{SE}}}
\ee 
for $x \neq y$ and $\QMH(x,x) = 1- \sum_{y \neq x} \QMH(x,y)$.

The main bound of this section is: 

\begin{lemma} \label{LemmaMainCompBound}
Fix $0 < r < \infty$. Let $\alpha(\QMH)$ and $1 - \beta_{1}(\QMH)$ be the log-Sobolev constant and spectral gap of $\QMH$. Then there exists some constant $C = C(c,r) < \infty$ that does not depend on $n$ so that 
\be 
\alpha(\QMH) &\geq \frac{C }{n^{2} \log(n)^{3} } \\
1 - \beta_{1}(\QMH) &\geq \frac{C}{n^{2} \log(n)^{2}}.
\ee  
\end{lemma}

\begin{proof} [Proof of Lemma \ref{LemmaMainCompBound}]

We will apply Theorem 2.1 of \cite{DiSa93b} (this result is equivalent to the special case of Theorem \ref{ThmDirGenChain} when $\Th = \hTh$, so we do not restate the result), comparing $\QMH$ to $\UMH$. Since $\Th = \hTh$, we do not need to define distributions or couplings; we need only define the relevant paths and flows on those paths.  The proof of this lemma will be similar in spirit to the proof of Lemma \ref{LemmaDirichletLaplaceMainCompBound}. In both cases:

\begin{enumerate}
\item there is an ``obvious" direct path between pairs of points $X,Y$;
\item  this ``obvious" path will sometimes leave the state space $\Omega_{n,k}$ of the Markov chain, and thus cannot legally be used; and
\item we resolve this problem by choosing intermediate points according to some distribution, and then showing that the indirect path from $X$ to $Y$ that goes via these intermediate points will stay in $\Omega_{n,k}$ with high probability.
\end{enumerate} 

The main difference between the lemmas is that the choice of measure from which to draw the intermediate points is more complicated in the present lemma. The result is an argument that is slightly longer but essentially the same. We now continue with the argument.

Fix $X,Y \in \Omega_{n,k}$ that satisfy $\UMH(x,y) > 0$. These are two configurations in $\Omega_{n,k}$ that satisfy $| \{ i \, : \, X[i] = 1\} \cap \{ i \, : \, Y[i] = 1\}| = k -1$. Let $x$ be the unique element of $\{i \, : \, X[i] = 1\} \backslash \{i \, : \, Y[i] = 1\}$ and let $y$ be the unique element of $\{i \, : \, X[i] = 1\} \backslash \{i \, : \, Y[i] = 1\}$. We will construct a random path from $X$ to $Y$ in two steps:

\begin{enumerate}
\item We construct short paths from $X$ and $Y$ to random configurations $X'$ and $Y'$ that are nearby but don't have any large clumps.
\item We construct a very simple path from $X'$ to $Y'$.
\end{enumerate}

More precisely, we write: 

\begin{defn} [Underlying Paths on $\LL$] \label{DefUnderlyingPath}
Throughout the remainder of this proof, we denote by $\Delta = \{ \delta_{x,y}\}_{x, y \in \LL}$ the collection of minimal-length paths between all pairs of points $x,y \in \LL$ that are described in Example 5.3 of \cite{DiSa93b}. We do not need the details of this collections of paths for our analysis, and so do not describe it further. These paths will be used to get between the intermediate vertices $X',Y'$ mentioned above.
\end{defn}

It is useful to write down notation for the neighbourhoods of the particles:

\begin{defn} [Small Covering]
Fix $X \in \Omega_{n,k}$ and $m \in \mathbb{N}$. We say that a disjoint collection of sets $C_{1},\ldots,C_{\ell} \subset \Lambda(L,2)$  is a size-$m$ covering of $X$ if:
\begin{enumerate}
\item Each set $C_{i},$ $i \in \{1,2,\ldots,\ell\}$, can be written as the (not-necessarily-disjoint) union of at most $k$ squares, each of side length exactly $m$.
\item The collection of sets satisfies $\{ u \, : \, \min_{v \in X} |u-v | \leq 2\} \subset \cup_{i=1}^{\ell} C_{i}$.
\item Each set $C_{i},$ $i \in \{1,2,\ldots,\ell\}$, contains at least one element of $X$.
\end{enumerate}
We call a a disjoint collection of squares a \textit{small covering} if it is a size-$m$ covering for some $m$.
\end{defn}

\begin{lemma} [Existence of Small Coverings]
Fix $k \in \mathbb{N}$, $m \geq 5$ odd and a size-$k$ subset $X$ of $\LL$. Assume that $L > 2m$. Then there exists a size-$m$ covering of $X$.
\end{lemma}

\begin{proof}
For $x \in X$, let $B_{x} = \{ u \, : \, \|u-v \|_{\infty} \leq \frac{m-1}{2}\}$. It is clear that $\{ u \, : \, \min_{v \in X} |u-v| \leq 2\} \subset \cup_{x \in X} B_{x}$, and that each set $B_{x}$ is a square of side length $m$. Merging any squares that intersect gives the desired covering.
\end{proof}

\begin{defn} [Sequence of Open Vertices]
Fix $\textbf{1}_{X} = \textbf{1}_{\{x_{1},\ldots, x_{k} \}} \in \Omega_{n,k}$, a \textit{privileged point} $x' \in \LL$ and small covering $C = (C_{1},\ldots,C_{\ell})$ of $X \cup \{x\}$. Say that a vertex $x_{i} \in X$ is \textit{open} if there is a path from $x_{i}$ to the boundary of $C$ that doesn't conflict with $X \cup \{x' \} \backslash \{x_{i}\}$ - that is, if there exists a sequence $y_{1},\ldots,y_{m} \in \LL$ with
\begin{enumerate}
\item $y_{1} = x_{i}$ and $y_{m} \notin \cup_{j} C_{j}$, 
\item $|y_{j+1} - y_{j} | =1$ for all $j \in \{1,2,\ldots,m-1\}$, and
\item $\min_{x \in X \cup \{x'\} \backslash \{x_{i}\}, \, 1 \leq j \leq m} |x - y_{j}| > 1$.
\end{enumerate}

With notation as above, we say that an ordering $\sigma \in S_{k}$ is \textit{a sequence of open vertices} if, for all $i \in \{1,2,\ldots,k-1\}$, $x_{\sigma[i]}$ is \textit{open} with respect to the configuration $X \backslash \cup_{j < i}\{ x_{\sigma[j]}\}$ and the same privileged points $x'$ and small covering $C$.
\end{defn}

\begin{lemma} [Existence of Sequence of Open Vertices] \label{LemmaExSeqOpen}
Fix $m,k \in \mathbb{N}$. Then for any $n > N(k,m) \equiv k^{4} m^{2}$ sufficiently large, any configuration $X  = \textbf{1}_{x_{1},\ldots,x_{k}} \in \Omega_{n,k}$, privileged point $x' \in \LL$, and size-$m$ covering $C$ has a sequence $\sigma$ of open vertices. 
\end{lemma}

\begin{proof}
We prove this by induction on $k$. For $k=1$, it is clear that this holds for any $n \geq  1$. Thus, it is sufficient to check that there always exists at least one open vertex. Define
\be 
M^{+} &= \max_{1 \leq i \leq k} x_{i}[1], \\
M^{-} &= \min_{1 \leq i \leq k} x_{i}[1], \\
z^{\pm} &\in \{1 \leq i \leq k \, : \, x_{i}[1] = M^{\pm} \}, \\
y_{j}^{\pm} &= (z^{\pm}[1] \pm (j-1), z^{\pm}[2]).
\ee 
Then at least one of $\{y_{j}^{+}\}_{j \geq 1}$, $\{y_{j}^{-}\}_{j \geq 1}$ is a path that satisfies the requirement in the definition of an open vertex for  $x'$, as long as the set $C$ does not cover any full line in $\LL$. It is clear that this last condition holds as long as $\sqrt{n} > k^{2} \, m$, so we have the result for $N(k,m) = k^{4} \, m^{2}$.
\end{proof}

\begin{defn} [Non-Interfering Locations]
For configurations $X, Y \subset \LL$, set $C \supset X,Y$ and points $x,y \in \LL$, we define the collection of \textit{non-interfering locations} as
\be 
\mathcal{N}(X,Y,x,y) = \{u \in \LL \, : \, \min_{v \in X \cup Y \cup C^{c} \cup \delta_{x,y}} |u -v| > 3\},
\ee 
where $\delta_{x,y}$ is as in Definition \ref{DefUnderlyingPath}.

\end{defn}

\begin{defn} [First Path Segment: Removing Clumps] \label{DefPathOne}
Fix a parameter $T \in \mathbb{N}$, write $X = \{x_{1},\ldots, x_{k}\}$, $Y = \{y_{1},\ldots,y_{k}\}$ and assume $| X \cap Y | = k -1$.  Let $p$, $q$ be the unique elements of $X \backslash Y$, $Y \backslash X$ respectively,  let $C$ be a small covering of $X \cup Y$ of size $\lceil 10 r^{4} \log(n)^{4} \rceil$, and let $\sigma^{(x)}$ (respectively $\sigma^{(y)}$) be a sequence of open vertices associated with the set $X \cap Y$, small covering $C$ of $X \cup Y$, and privileged point $p$ (respectively $q$); we choose $\sigma^{(x)}, \sigma^{(y)}$ uniformly at random among these open sequences. Note that, by Lemma \ref{LemmaExSeqOpen}, there is always at least one choice for each of $\sigma^{(x)}$, $\sigma^{(y)}$.

We then define a pair of measures $F_{1}^{X,Y}$ and $F_{1}^{Y,X}$ on paths started from $\textbf{1}_{X}$ and $\textbf{1}_{Y}$ respectively. We will not construct these two marginal distributions themselves; instead, we define a joint distribution on paths $(P_{X}, P_{Y})$ with $P_{X} \sim F_{1}^{X,Y}$ and $P_{Y} \sim F_{1}^{Y,X}$. Note that the following algorithm builds up its paths over $k-1$ distinct stages:

\begin{enumerate}
\item Set $X^{(1)} = X \backslash \{ x_{\sigma[1]} \}$, $Y^{(1)} = Y \backslash \{ y_{\sigma[1]} \}$.
\item For $i \in \{1,2,\ldots,k-1\}$,
\begin{enumerate}
\item Let $\{Z_{t}^{(i)}\}_{t=1}^{T}$, $\{\hat{Z}_{t}^{(i)}\}_{t=1}^{T}$ be Metropolis-Hastings chains, with proposal given by $\frac{1}{2}$-lazy simple random walk on $\LL$ and target distributions being uniform on $C \backslash \{ u \, : \, \min_{v \in X^{(i)}}  | u - v | \leq 1\}$ and  $C \backslash \{ u \, : \, \min_{v \in Y^{(i)}}  | u - v | \leq 1\}$ respectively. Let the initial points of these chains be $Z_{1}^{(i)} = x_{\sigma[i]}$ and $\hat{Z}_{1}^{(i)} =  y_{\sigma[i]}$ respectively. Couple these two chains so as to maximize $\P[Z_{T}^{(i)} = \hat{Z}_{T}^{(i)}]$.
\item If $\hat{Z}_{T}^{(i)} = Z_{T}^{(i)} \in \mathcal{N}(X^{(i)}, Y, p, q) \cap \mathcal{N}(Y^{(i)}, X, p, q)$, define the $i$'th part of the path by setting $\gamma_{X}(i)' = ( X^{(i)} \cup Z_{1}^{(i)}, X^{(i)} \cup Z_{2}^{(i)}, \ldots, X^{(i)} \cup Z_{T}^{(i)} )$ and $\gamma_{Y}(i)' = ( Y^{(i)} \cup \hat{Z}_{1}^{(i)}, Y^{(i)} \cup \hat{Z}_{2}^{(i)}, \ldots, Y^{(i)} \cup \hat{Z}_{T}^{(i)} )$, and letting $\gamma_{X}(i)$, $\gamma_{Y}(i)$ be obtained by removing repeated elements of $\gamma_{X}(i)'$, $\gamma_{Y}(i)'$. Otherwise, say that step $i$ \textit{failed} and return to step \textbf{(2.a)}. 
\end{enumerate}
\item Return the paths $(\gamma_{X}(1), \gamma_{X}(2),\ldots,\gamma_{X}(k-1))$ and $(\gamma_{Y}(1), \gamma_{Y}(2),\ldots,\gamma_{Y}(k-1))$.
\end{enumerate}

We denote by $X' = \gamma_{X}(k-1)$ and $Y' = \gamma_{Y}(k-1)$ the random endpoints of these paths.
\end{defn}

\begin{defn} [Second Path Segment: Matching Elements] \label{DefPathTwo}
We define a flow $F_{2}^{X,Y}$. Fix $X,Y$ satisfying $|X| = |Y| = k$ and $|X \cap Y | = k-1$. Let $p,q$ be the unique elements of $X \backslash Y$ and $Y \backslash X$. Let $\delta_{p,q} = (z_{1},\ldots,z_{m})$ be as in Definition \ref{DefUnderlyingPath}. For $1 \leq i \leq m$, define $Z_{i} = (X \cap Y) \cup \{z_{i}\}$. Define $\gamma_{X,Y} = (\textbf{1}_{Z_{1}}, \ldots, \textbf{1}_{Z_{m}})$. When every element of $\gamma_{X,Y}$ is an element of $\Omega_{n,k}$, $F_{2}^{X,Y}$ assigns weight 1 to $\gamma_{X,Y}$. Otherwise, we do not define $F_{2}^{X,Y}$.
\end{defn}

Finally, we define a measure $F$ on $\Gamma_{X,Y}$ by giving an algorithm for sampling from $F$:

\begin{defn} [Full Path] \label{DefFullPath}
Set $T = 2 \log(n)^{10} \log(\log(n))$. To sample from $F$, run the following random algorithm:
\begin{enumerate}
\item Sample paths $P_{X} \sim F_{1}^{X,Y}$, $P_{Y} \sim F_{1}^{Y,X}$ with endpoints $X'$, $Y'$ according to the coupling in Definition \ref{DefPathOne}.
\item Sample a path $P_{X',Y'} \sim F_{2}^{X',Y'}$. When this path is not defined, say that \textit{the long path fails} and return to Step \textbf{(1)}. 
\item Return the random path $(P_{X}, P_{X',Y'}, P_{Y}^{\dag})$, where $P_{Y}^{\dag}$ denotes reversing the order of a path.
\end{enumerate}
\end{defn}

Having defined the flows, we have implicitly defined the constant $\mathcal{A}$ in Theorem  \ref{ThmDirGenChain}. We must now bound that constant. To do so, we consider a fixed edge $(q,r) \in E_{\mathrm{MH}}$ and bound the total weight of all paths that cross through $(q,r)$. Since the constant $\mathcal{A}$ is defined as a sum over all flows, we can bound the contributions due to the first type of path (see Definition \ref{DefPathOne}) and the second type of path (see Definition \ref{DefPathOne}) separately.

We begin by bounding the flow due to the first type of path. First, we show that with probability $1 - o(1)$, none of the $k-1$ steps in the construction of $F_{1}^{X,Y}$ will fail, and also the long path obtained following the initial sampling from $F_{1}^{X,Y}$ will not fail. Checking that, with overwhelming probability, \textit{none} of the events fail will allow us to essentially ignore the rejection steps when estimating the weight given to any edge, at the cost of a small multiplicative constant. This substantially simplifies our analysis.

\begin{lemma} [Local Failures Are Rare] \label{LemmaGoodConf}

Following the notation of Definition \ref{DefPathOne}, there exists $N_{0} = N_{0}(r)$ so that the probability $\P_{i}$ that step $i$ \textit{fails} is bounded by
\be 
\P_{i} \leq \frac{1}{2(r \log(n))^{1.5}}
\ee 
for all $T > \log(n)^{10} \log(\log(n))$ and all $n > N_{0}$ sufficiently large. 
\end{lemma}

\begin{proof}

Let $\{Z_{t}\}_{t \geq 1} = \{Z_{t}^{(i)}\}_{t \geq 1}$ be as in stage $i$ of Definition \ref{DefPathOne}, and let $\mathcal{D}$ be the connected component of its state space,  $C \backslash \{ u \, : \, \min_{v \in X^{(i)}}  | u - v | \leq 1\}$, that contains $Z_{1}$. The critical estimate is the following mixing bound:

\begin{prop} \label{PropMixingSrwWithHoles}
The mixing time $\tmix$ of $\{Z_{t}\}_{t \geq 1}$ on $\mathcal{D}$ satisfies 
\be \label{IneqMixingOfHolyWalk}
\tmix \leq C_{1} \log(n)^{8} \log(\log(n))
\ee 
for some $0 < C_{1} = C_{1}(r) < \infty$.
\end{prop}
\begin{proof}
We note that $| \mathcal{D}| = O(r^{8} \log(n)^{8})$, and all non-zero transition probabilities for $\{Z_{t}\}_{t \geq 1}$ are at least $\frac{1}{4}$. By Theorem 1 of \cite{morris2005evolving}, we have 
\be 
\tmix \leq 64 | \mathcal{D}|^{2}  (\log(|\mathcal{D}|) + \log(4)) = O( r^{8} \log(n)^{8} \log(r \, \log(n))).
\ee 
\end{proof}

Having proved Proposition \ref{PropMixingSrwWithHoles}, we now continue with the proof of Lemma \ref{LemmaGoodConf}.

Let the good set $\mathcal{N} =  \mathcal{N}(X^{(i)}, Y^{(i)}, p, q)$ and small covering $C = (C_{1},\ldots,C_{\ell})$ be as in stage $i$ of Definition \ref{DefPathOne}. Assume that $Z_{1}$ is in the element $C_{1}$ of the small covering. For a collection of points $A \subset \LL$ and $j \in \mathbb{N}$, denote by $\partial_{j} A = \{ u \in C_{1} \, : \, \min_{v \in A} |u - v| \leq j\}$. For a collection of points $A \subset \LL$, let $W(A)$ be the complement of the connected component of $C_{1} \backslash A$ that contains all elements of $(\partial_{1} C) \cap (C_{1} \backslash A)$, when such a connected component exists. We note by the isoperimetric inequality for squares (see \textit{e.g.} Theorem 1.2 of \cite{hamamuki2014discrete}) that 
\be 
| W(A)| \leq |A|^{2} (4)^{-2}.
\ee

We then have 

\be \label{IneqSizeOfGoodSet}
\frac{|\mathcal{N}|}{|\mathcal{D}|} &\geq \frac{|C_{1}| - | \partial_{3} C_{1}| - W(\partial_{3}(\delta_{p,q} \cup X^{(i)} \cup_{j < i} Z_{T}^{(j)}))}{|C_{1}|} \\
&\geq \frac{|C_{1}| - 12 \, r\, \log(n) \lceil 10 r^{4} \log(n)^{4} \rceil - W(\partial_{3}(\delta_{p,q} \cup X^{(i)} \cup_{j < i} Z_{T}^{(j)}))}{|C_{1}|} \\
&\geq \frac{|C_{1}| -12 \, r\, \log(n) \lceil 10 r^{4} \log(n)^{4} \rceil - (4)^{-2} |\partial_{3}(\delta_{p,q} \cup X^{(i)} \cup_{j < i} Z_{T}^{(j)})|^{2} }{|C_{1}|} \\
&\geq \frac{|C_{1}| - 12 \, r\, \log(n) \lceil 10 r^{4} \log(n)^{4} \rceil - (4)^{-2} (14 \, \sqrt{| C_{1}|} + 7^{2} \, (r \log(n)) )^{2} }{|C_{1}|} \\
&= 1 - O( ( r \log(n))^{-2}). \\
\ee 

Combining Inequalities \eqref{IneqMixingOfHolyWalk} and \eqref{IneqSizeOfGoodSet}, there exists some constant $A = A(r)$ so that for $T > A \, \log(n)^{8} \log(\log(n))$ and $n$ sufficiently large,
\be 
\P[Z_{T} \in \mathcal{N}] &\geq \frac{|\mathcal{N}|}{|\mathcal{D}|} - 2 ( 2^{-\lfloor \frac{T}{\tmix} \rfloor + 1} )\\
&\geq 1 - O(\frac{1}{(r \log(n))^{1.5}}),
\ee 
completing the proof.
\end{proof}

The same argument (with easier estimates) bounds the probability of rejecting a full path:

\begin{lemma} [Long Paths Rarely Fail] \label{LemmaLongPathNonFailure}
Following the notation of Definition \ref{DefFullPath}, for the constants $T,r$ as in Lemma \ref{LemmaGoodConf}, we have 
\be 
\P[\text{the long path fails}] = o(1)
\ee 
as $n$ goes to infinity.
\end{lemma}

\begin{proof}
Let $X,Y$ and $p,q$ and $X',Y'$ be as in Definition \ref{DefFullPath}, and define $\mathcal{N} = \mathcal{N}(X^{(k-1)}, Y^{(k-1)}, p, q)$. Using the notation of Definition \ref{DefFullPath}, we note that the path $\delta_{p,q} = (z_{1},\ldots,z_{m})$ depends only on the two points $p, q \in X \Delta Y$, not any further randomization. Furthermore, $\delta_{p,q}$ is a minimal-length path between $p$ and $q$, and thus its intersection $|\delta_{p,q} \cap \mathcal{N} \cap \mathcal{C}|$ with the roughly-square set $\mathcal{C} \cap \mathcal{N}$ is of size $O(\sqrt{|\mathcal{N} \cap \mathcal{C}|}) =O(r^{4}\log(n)^{4})$. Therefore, by the same calculation as in Inequality \eqref{IneqSizeOfGoodSet}, we have 
\be 
\frac{| \delta_{p,q} \cap \mathcal{N} \cap \mathcal{C}|}{|\mathcal{N} \cap \mathcal{C}|} = O\left( \frac{1}{r^{4} \log(n)^{4}} \right).
\ee 
Combining this with Inequality \eqref{IneqMixingOfHolyWalk} completes the proof.
\end{proof}

Next, we show that this implies the total weight given to any particular edge is small:

\begin{lemma} [Contribution of First Path Type] \label{LemmaPathTypeOne}
Following the notation of Definition \ref{DefPathOne} and fixing $T \geq  \log(n)^{10} \log(\log(n))$ so that Lemmas \ref{LemmaGoodConf} and \ref{LemmaLongPathNonFailure} apply, we have for all distinct $A,B \in \Omega_{n,k}$ satisfying $\QMH(A,B) > 0$ that 
\be 
\sum_{X,Y \, : \, | X \cap Y| = k-1} \sum_{\gamma \in \Gamma_{X,Y} \, : \, (A,B) \in \gamma} F_{1}^{X,Y}[\gamma] \leq 8 n k^{2} (T+1).
\ee 
\end{lemma}

\begin{proof}

Fix a configuration $A \in \Omega_{n,k}$, choose $X \sim \mathrm{Unif}(\Omega_{n,k})$ and then choose $Y \sim \mathrm{Unif}(\{ y \in \Omega_{n,k} \, : \, |X \cap y| = k-1\})$, and choose a random path $\gamma_{X} = (\gamma_{X}(1), \gamma_{X}(2),\ldots,\gamma_{X}(k-1))$ according to Definition \ref{DefPathOne}. For $1 \leq i \leq k-1$, note that we can write $\gamma_{X}(i) = ( X^{(i)} \cup Z_{1}^{(i)}, X^{(i)} \cup Z_{2}^{(i)}, \ldots, X^{(i)} \cup Z_{T}^{(i)} )$ as in that definition. Using this notation, 

\be \label{IneqHittingFlowBound}
\sum_{X,Y \, : \, | X \cap Y| = k-1} \sum_{\gamma \in \Gamma_{X,Y} \, : \, (A,B) \in \gamma} F_{1}^{X,Y}[\gamma] \leq n \, {n \choose k} \sum_{i=1}^{k-1} \sum_{t=0}^{T} \P[A = X^{(i)} \cup Z_{t}^{(i)}].
\ee 

Thus, it is enough to bound the probabilities $\P[A = X^{(i)} \cup Z_{t}^{(i)}]$ for all $ 1 \leq i \leq k-1$ and $0 \leq t \leq T$. We begin by bounding these probabilities in the special case $i = 1$. Noting that all particles in $X$ are open with probability $1 - o(1)$, we have by the usual `birthday problem' bound that
\be \label{IneqAnnoyingPathSillyBound1}
\P[X^{(1)} = S] \leq \frac{2}{ {n \choose k-1}}
\ee 
for all $S \in \Omega_{n,k-1}$ and all $n > N_{0}(r)$ sufficiently large. 

Next, we note that the following is a valid rejection-sampling algorithm for choosing $x_{\sigma[1]}$ conditional on $X^{(1)}$:
\begin{enumerate}
\item Sample $\hat{x}$  uniformly at random from among all elements of the largest connected component $\mathcal{D}_{1}$ of $\LL \backslash X^{(1)}$.
\item Let $p = p(\hat{x},X^{(1)})$ be the percentage of all sequences of open vertices for configuration $X^{(1)} \cup \{\hat{x}\}$ that begin with $\hat{x}$. Then \textit{accept} $\hat{x}$ with probability $p$; otherwise \textit{reject} and go back to step \textbf{(1)}.
\end{enumerate}

Before analyzing the ``corrected" choice of $x_{\sigma[1]}$ conditional on $X^{(1)}$, we analyze the ``uncorrected" choice of $\hat{x}$. Sample $\hat{x}$ (conditional on $X^{(1)}$) uniformly at random from among all elements of the largest connected component $\mathcal{D}_{1}$ of $\LL \backslash X^{(1)}$. Let $\{\hat{Z}_{t}^{(1)} \}_{t \in \mathbb{N}}$ be the Markov chain on $\mathcal{D}_{1}$ constructed as in Step \textbf{(2.a)} of Definition \ref{DefPathOne}, started at $\hat{Z}_{1}^{(1)} = \hat{x}$. Since $\hat{x}$ was drawn from the stationary measure of $\{\hat{Z}_{t}^{(1)}\}_{t \in \mathbb{N}}$, we have

\be \label{IneqAnnoyingPathSillyBound2Uncorrected}
\P[\hat{Z}_{t}^{(1)} = z] = \frac{1}{|\mathcal{D}_{1}|}
\ee 
for all $z \in \mathcal{D}_{1}$ and all $t \in \mathbb{N}$. By the above rejection-sampling algorithm for $x_{\sigma[1]}$ and the obvious bounds $\frac{1}{k} \leq p \leq 1$, this implies that the true path $\{Z_{t}^{(1)}\}_{t \in \mathbb{N}}$ that appears in Step \textbf{(2.a)} of Definition \ref{DefPathOne} satisfies
\be \label{IneqAnnoyingPathSillyBound2}
\P[Z_{t}^{(1)} = z] \leq \frac{k}{|\mathcal{D}_{1}|}
\ee 
for all $z \in \mathcal{D}_{1}$ and $t \in \mathbb{N}$. Combining Inequalities \eqref{IneqAnnoyingPathSillyBound1} and \eqref{IneqAnnoyingPathSillyBound2}, with the bound \eqref{IneqSizeOfGoodSet} on $|\mathcal{D}_{1}|$, we conclude that 
\be \label{IneqMiniconclusion1}
\P[X^{(1)} \cup Z_{t}^{(1)} = S] \leq \frac{4k}{ {n \choose k}}
\ee 
for all $S \in \Omega_{n,k}$ and all $0 \leq t \leq T$, whenever $n > N_{0}(r)$ is sufficiently large.

Analogous bounds for $1 < i \leq k-1$ will follow by Proposition \ref{PropMixingSrwWithHoles}. In particular, let $\mathcal{D}_{i}$ be the largest connected component of $\LL \backslash X^{(i)}$. By Proposition \ref{PropMixingSrwWithHoles} and the same argument giving Inequality \eqref{IneqAnnoyingPathSillyBound1}, we have 

\be \label{IneqAnnoyingPathSillyBound3}
\P[X^{(i)} = S] \leq \frac{2}{ {n \choose k-1}} + i \, n^{-10}
\ee 
for all $S \in \Omega_{n,k-1}$, all $1 \leq i \leq k-1$ and all $n < N_{0}(r)$ sufficiently large. Similarly, by Proposition \ref{PropMixingSrwWithHoles} and the same argument giving Inequality \eqref{IneqMiniconclusion1}, we have
\be \label{IneqMiniconclusion2}
\P[X^{(i)} \cup Z_{t}^{(i)} = S] \leq k \left( \frac{4}{ {n \choose k}} + 2i \, n^{-10}\right)
\ee 
for all $S \in \Omega_{n,k}$, all $1 \leq i \leq k-1$ and all $n < N_{0}(r)$ sufficiently large.

Combining Inequalities \eqref{IneqHittingFlowBound} and \eqref{IneqMiniconclusion2} and applying Lemma \ref{LemmaLongPathNonFailure} completes the proof of the lemma.
\end{proof}

 The following bound on the contribution of the second path type follows immediately from Example 5.3 of \cite{DiSa93b} and Lemmas \ref{LemmaGoodConf} and \ref{LemmaLongPathNonFailure}:

\begin{lemma} [Contribution of Second Path Type] \label{LemmaPathTypeTwo}
For $X,Y \in \Omega_{n,k}$ satisfying $| X \cap Y| = k-1$, let the random variables $X' = X'(X,Y)$ and $Y' = Y'(X,Y)$ be as in Definition \ref{DefPathOne}.
Following the notation of Definition \ref{DefPathTwo} and fixing $T \geq  \log(n)^{10} \log(\log(n))$, we have for all distinct $A,B \in \Omega_{n,k}$ satisfying $\QMH(A,B) > 0$ that 
\be 
\sum_{X,Y \, : \, | X \cap Y| = k-1} \sum_{x',y' \in \Omega_{n,k}} \P[ (X',Y') = (x',y') | X,Y] \sum_{\gamma \in \Gamma_{x',y'} \, : \, (A,B) \in \gamma} F_{2}^{x',y'}[\gamma] \leq 16 n^{1.5}.
\ee 
\end{lemma}

\begin{proof}
Fix a configuration $A \in \Omega_{n,k}$ and a pair  $X,Y \in \Omega_{n,k}$ satisfying $| X \cap Y| = k-1$. Let $x = x(X,Y)$ and $y = y(X,Y)$ be the unique elements of $X \backslash Y$ and $Y \backslash X$ respectively. By Lemmas \ref{LemmaGoodConf} and \ref{LemmaLongPathNonFailure}, 
\be 
\P[(X',Y') = S \, | \, x(X,Y), \, y(X,Y)] \leq \frac{4}{ {n \choose k}^{2}}
\ee 
for all $S \subset \Omega_{n,k}^{2}$ and all $n > N_{0}(r)$ sufficiently large. In particular, the probability mass function of $(X',Y')$ conditional on $x(X,Y)$ and $y(X,Y)$ is bounded by a constant factor times the probability mass function of the uniform distribution.

By the same calculation as in Example 5.3 of \cite{DiSa93b}, this implies
\be 
\P[A \in P_{X',Y'}] \leq \frac{16 \sqrt{n}}{{ n \choose k}}
\ee 
for all $n > N_{0}(r)$ sufficiently large. The result follows immediately by the same bound as Inequality \eqref{IneqHittingFlowBound}.

\end{proof}

Combining Lemmas \ref{LemmaPathTypeOne} and \ref{LemmaPathTypeTwo}, and noting that all paths have length at most $k(T+1) + 2 \sqrt{n}$, we have

\be 
\mathcal{A} &\leq  (\max_{x,y,q,r \, : \, \QMH(q,r) > 0} \frac{\UMH(x,y)}{\QMH(q,r)}) \times (\max_{\gamma \, : \, F(\gamma) > 0}  |\gamma| ) \times ( \max_{(q,r) \, : \, \QMH(q,r) >0}  \sum_{\gamma \ni (q,r)} F[\gamma] ) \\
&\leq (\frac{2}{n}) \times ( k(T+1) + 2 \sqrt{n}) \times (8 n k^{2} (T+1) + 16 n^{1.5})
\ee  
for all $n > N(c,r)$ sufficiently large.

Lemma \ref{LemmaMainCompBound} now follows immediately from an application of Theorem \ref{ThmDirGenChain}, with comparison provided by Lemma \ref{LemmaDirichletLaplaceMainCompBound}.

\end{proof}

\subsection{Comparison of Modified Simple Exclusion Process to KCIP} \label{SubsecSeKcip}

Let $\alpha(Q_{n,k}), 1 -\beta_{1}(Q_{n,k})$ be the log-Sobolev constant and spectral gap of $Q_{n,k}$, and let $\alpha(Q_{\mathrm{MH}})$ and $1 - \beta_{1}(Q_{\mathrm{MH}})$ be the log-Sobolev constant of $Q_{\mathrm{MH}}$. As shown in Inequality (5.10) of \cite{pillai2015mixing}, we have:

\be \label{IneqQuotingIneq510}
\alpha(Q_{n,k}) &\geq \frac{1}{4n} \alpha(Q_{\mathrm{MH}}) \\
1 - \beta_{1}(Q_{n,k}) &\geq \frac{1}{4n} (1 - \beta_{1}(Q_{\mathrm{MH}})).
\ee

\subsection{Proof of Lemma \ref{CorMixingTimeRestWalk}}

We put together the bounds obtained in Sections \ref{SubsecCompDirichletLaplace}, \ref{SubsecCompDirichletLaplaceSE} and \ref{SubsecSeKcip}:

\begin{proof} [Proof of Lemma \ref{CorMixingTimeRestWalk}]
By Lemma \ref{LemmaMainCompBound} and Inequality \eqref{IneqQuotingIneq510}, the log-Sobolev constant $\alpha(Q_{n,k})$ and spectral gap $1 - \beta_{1}(Q_{n,k})$ of $Q_{n,k}$ satisfy 
\be 
\alpha(Q_{n,k}) &\geq \frac{C}{n^{3} \log(n)^{3}} \\
1 - \beta_{1}(Q_{n,k}) &\geq \frac{C}{n^{3} \log(n)^{2}}
\ee 
for some $C = C(c,r)$ for all $n > N(c,r)$ sufficiently large. Lemma \ref{CorMixingTimeRestWalk} follows immediately from an application of Inequality (3.3) of \cite{DiSa96c}.

\end{proof}

\section{Mixing at Moderate Density: Main Bounds} \label{SecMixingMod}

We fix some notation for the remainder of this section. For fixed $1 \leq k \leq n$, let $\{Y_{t}\}_{t \in \mathbb{N}}$ be the trace of $\{X_{t}\}_{t \in \mathbb{N}}$ on $\cup_{i=1}^{k} \Omega_{n,i}$, let $P_{n,k}$ be the transition kernel of $\{ Y_{t}\}_{t \in \mathbb{N}}$, and let $\tau_{\mathrm{mix}}^{(\leq k)}$ be the mixing time of $P_{n,k}$. Our main result is:

\begin{lemma} [Mixing at Moderate Density] \label{LemmaMixingModDensity}
Fix $ 0 < r < \infty$. There exists a constant $C = C(r,c)$ so that 
\be 
\tau_{\mathrm{mix}}^{(\leq k)} \leq C n^{3} \log(n)^{13}
\ee 
uniformly in $1 \leq k \leq r \log(n)$.
\end{lemma}

Our main strategy is to use Theorem 1.1 of \cite{MaRa02}, along with some soft bounds, to `glue together' the bounds on $\{ \tau_{\mathrm{mix}}^{(k)} \}_{1 \leq k \leq r \log(n)}$ from Section \ref{SecCompToExc}. The basic idea of \cite{MaRa02} (as well as recent related papers \cite{JSTV04}, \cite{PiSm15}) is that it is possible to bound the relaxation time of a Markov chain on a state space $\Theta$ decomposed as $\Theta = \cup_{i=1}^{m} \Theta_{i}$ in terms of the relaxation times of certain ``restricted" chains on $\Theta_{1},\ldots,\Theta_{m}$ and a ``projected" chain on $\{1,2,\ldots,m\}$ that measures the typical transition rates between $\Theta_{1},\ldots,\Theta_{m}$ near stationarity. In our case, we have obtained bounds on the relevant ``restricted" chains in Section \ref{SecCompToExc}, and we will be able to easily compare our ``projected" chain to biased random walk on the path $\{1,2,\ldots,k-1\}$.

\subsection{Review of Results in \cite{MaRa02}} \label{SecReviewPart}
Fix $k \in \mathbb{N}$. Let $P$ be the transition kernel of the KCIP $\{X_{t}\}_{t \in \mathbb{N}}$, and for $1 \leq i \leq k-1$, let $P_{i}$ be the \textit{restriction} of $P$ to $\Omega_{i} \cup \Omega_{i+1}$, defined by:
\be 
P_{i}(x,y) =
\begin{cases}
 P(x,y), \, \, \qquad \qquad \qquad \qquad x \neq y, \, x,y \in \Omega_{i} \cup \Omega_{i+1} \\
 1 - \sum_{y \in \Omega_{i} \cup \Omega_{i+1}} P(x,y), \qquad x=y \\
 0,  \, \, \, \, \, \qquad \qquad \qquad \qquad  \qquad \qquad \text{otherwise.}
 \end{cases}
\ee 
Also define the kernel $\tilde{P}$ on the discrete set $\{1,2,\ldots, k-1\}$ by 
\be 
\tilde{P}(i,j) &= \frac{ \pi( (\Omega_{i} \cup \Omega_{i+1}) \cap (\Omega_{j} \cup \Omega_{j+1}))}{ 3 \pi( \Omega_{i} \cup \Omega_{i+1})}, \qquad i \neq j \\
\tilde{P}(i,i) &= 1 - \sum_{j \neq i} \tilde{P}(i,j).
\ee 
Theorem 1.1 of \cite{MaRa02} implies:
\begin{thm} \label{ThmQuoteRandallMadras}
With notation as above, 
\be 
1 - \beta_{1}(P) \geq \frac{1}{9} (1 - \beta_{1}(\tilde{P})) \, \min_{1 \leq i \leq k-1} (1 - \beta_{1}(P_{i})).
\ee 
\end{thm}

\subsection{Bounds on $P_{i}$ and $\tilde{P}$} \label{SecBoundCompSE}

We obtain bounds on the spectral gaps of the kernels $\{ P_{i} \}_{i=1}^{k-1}$ and $\tilde{P}$ defined in Section \ref{SecReviewPart}. We begin by bounding $1 - \beta_{1}(\tilde{P})$:

\begin{lemma} \label{LemmaIneqTildeGap} 
Fix $0 < r < \infty$. Then there exists $C = C(r,c)$ so that 
\be 
1 - \beta_{1}(\tilde{P}) \geq \frac{C}{\log(n)^{2}}
\ee 
uniformly in $1 \leq k \leq r \log(n)$. 
\end{lemma}
\begin{proof}
The proof involves first obtaining bounds for the hitting time of a reversible Markov chain $Z_t$ evolving according to $\tilde{P}$. Then we use Theorem 1.1 of \cite{PeSo13} to obtain a mixing time estimate for $\tilde{P}$ from our bound on its hitting times.\par
We can assume without loss of generality that $k = \lfloor r \log(n) \rfloor$. 
We begin by expanding our formula for $\tilde{P}$. For $1 \leq i < k$, the usual `birthday problem' bound gives
\be 
\tilde{P}(i,i+1)&= \frac{ \pi( \Omega_{i+1}) }{ 3 \pi( \Omega_{i} \cup \Omega_{i+1})} \\
&= \frac{1}{3} \frac{ | \Omega_{i+1} | (\frac{c}{n})^{i+1} (1 - \frac{c}{n})^{n-i-1}}{ | \Omega_{i} | (\frac{c}{n})^{i} (1 - \frac{c}{n})^{n-i} +  | \Omega_{i+1} | (\frac{c}{n})^{i+1} (1 - \frac{c}{n})^{n-i-1} } \\
&= \frac{1}{3} \frac{c}{i+1+c} (1 + O(\frac{r^{2} \log(n)^{2}}{n})).
\ee 
Similarly, for $1 < i \leq k$,
\be 
\tilde{P}(i,i-1)&= \frac{ \pi( \Omega_{i}) }{ 3 \pi( \Omega_{i} \cup \Omega_{i+1})} \\
&= \frac{1}{3} \frac{ | \Omega_{i} | (\frac{c}{n})^{i} (1 - \frac{c}{n})^{n-i}}{ | \Omega_{i} | (\frac{c}{n})^{i} (1 - \frac{c}{n})^{n-i} +  | \Omega_{i+1} | (\frac{c}{n})^{i+1} (1 - \frac{c}{n})^{n-i-1} } \\
&= \frac{1}{3} \frac{i+1}{i+1+c}(1 + O(\frac{r^{2} \log(n)^{2}}{n})).
\ee 
Finally, for all  $1 \leq i \leq k$,
\be 
\tilde{P}(i,i) &= 1 - \tilde{P}(i,i+1) - \tilde{P}(i,i-1) \\
&\geq \frac{1}{3} (1 + O(\frac{r^{2} \log(n)^{2}}{n})),
\ee 
where by convention $\tilde{P}(1,0) = \tilde{P}(k,k+1) = 0$. Similarly, the stationary distribution $\tilde{\mu}$ of $\tilde{P}$ satisfies
\be 
\tilde{\mu}(i) &= \frac{1}{Z} {n \choose i} (\frac{c}{n})^{i} (1 - \frac{c}{n})^{n-i} (1 + O(\frac{r^{2} \log(n)^{2}}{n})) \\
Z &= \sum_{j=1}^{k} {n \choose i} (\frac{c}{n})^{i} (1 - \frac{c}{n})^{n-i} (1 + O(\frac{r^{2} \log(n)^{2}}{n})). 
\ee

Let $m^{-} = \max(1,\lfloor \frac{c}{4} \rfloor)$, $m^{+} = \lceil 4 \max(1,c) \rceil$. For $\{Z_{t}\}_{t \in \mathbb{N}}$ a Markov chain with transition kernel $\tilde{P}$, let 
\be 
\tau^{-} &= \min \{ t \in \mathbb{N}  \, : \, Z_{t} = m^{-} \} \\
\tau^{+} &= \min \{ t \in \mathbb{N}  \, : \, Z_{t} = m^{+} \} 
\ee be the first hitting times of $m^{-}$ and $m^{+}$ respectively. By standard formulas (see \textit{e.g.,} \cite{PaTe96}),
\be \label{IneqHitPlusUpper}
\E[\tau^{+} | Z_{1} = 1] - 1 &= \sum_{v=1}^{m^{+}-1} \Big( \frac{1}{\tilde{\mu}(v) \tilde{P}(v,v+1)} \, \sum_{q=1}^{v} \tilde{\mu}(q) \Big) \\
&= \big(1 + O(\frac{r^{2} \log(n)^{2}}{n})\big) \frac{1}{3c}  \sum_{v=1}^{m^{+}-1} \Big( \frac{v+1+c}{ {n \choose v} (\frac{c}{n})^{v} } \sum_{q=1}^{v}  {n \choose q} (\frac{c}{n})^{q} \Big) \\
&\leq (1 + O(\frac{r^{2} \log(n)^{2}}{n})) \frac{1}{3c} \sum_{v=1}^{m^{+}-1} v(v+1+c) \\
&= O(  (m^{+})^{2}) = O(k^{2}).
\ee 
Using the same argument we can also obtain that
\be \label{IneqHitMinUpper}
\E[\tau^{-} | Z_{1} = k] = O( k^{2}).
\ee 

We also note that $\sum_{v=1}^{m^{+}} \tilde{\mu}(v) > 0.501$, $\sum_{v=m^{-}}^{k} \tilde{\mu}(v) > 0.501$ for all $n$ sufficiently large. Combining this fact with Inequalities \eqref{IneqHitPlusUpper} and \eqref{IneqHitMinUpper}, Theorem 1.1 of \cite{PeSo13} implies that the mixing time $\tilde{\tau}_{\mathrm{mix}}$ of $\tilde{P}$ satisfies 
\be 
\tilde{\tau}_{\mathrm{mix}} = O(k^{2}).
\ee 
Since the mixing time of a Markov chain bounds its relaxation time, this completes the proof.
\end{proof}

Next, we will bound the spectral gap of $P_{i}$. This will follow from the bounds in Section \ref{SecCompToExc} on the bounds of the trace processes on $\Omega_{i}$ and $\Omega_{i+1}$, combined with some very basic bounds on the transition time between $\Omega_{i}$ and $\Omega_{i+1}$. The remainder of this section is devoted to computing these basic bounds. Although we give additional details in the proofs, the remainder of the arguments in Section \ref{SecBoundCompSE} are based on the following observations:

\begin{enumerate}
\item It is straightforward to check that, whenever $X_{t} \in \Omega_{i}$ contains two particles that are within distance 3 of each other, the probability of moving  from $\Omega_{i+1}$ to $\Omega_{i}$ within $O(n^{2})$ steps is bounded away from 0.
\item It is possible to couple the trace of $\{X_{t}\}_{t \in \mathbb{N}}$ on $\Omega_{i}$ to a $(1 - \frac{1}{n})$-lazy version of the simple exclusion process $\{Y_{t}\}_{t \in \mathbb{N}}$ with $i$ particles so that, with high probability, $X_{t} = Y_{t}$ until the first time that any two particles get within distance 3. We call such a time a ``near-collision time."
\item The rate of ``near-collision times" associated with the simple exclusion process are very well-understood (see \textit{e.g.} \cite{Cox89} and \cite{Oliv12b}).
\end{enumerate}

This means that we can bound the transition times between $\Omega_{i}$ and $\Omega_{i+1}$ by translating existing results on the simple exclusion process. The coupling mentioned in item (2) of the above sequence of observations is the obvious step-by-step maximal coupling, and so we do not give an explicit construction. Such an explicit construction is available in Section 7 of \cite{pillai2015mixing}.

To obtain the required bounds, we first recall some facts about the simple exclusion process. For $i \in \mathbb{N}$, let 
\be
\mathcal{H}(i) = \{ X \in \Omega_{i} \, : \, \min_{u,v \, : \, X[u]X[v] = 1} |u-v| > \frac{\sqrt{n}}{\log(n)^{0.25}}\}
\ee 
be the collection of very well-spaced configurations in $\Omega_{i}$, and define 
\be \label{DefWellSpaced}
\mathcal{G}(i) = \{Y \in \Omega_{i+1} \, : \, \exists \, X \in \mathcal{H}, \, x \in \LL \text{ s.t. } Y = X \cup \{x \}, \, \min_{u \, : \, X[u] = 1} |u-x| = 2 \}
\ee 
to be the collection of all configurations in $\Omega_{i+1}$ consisting of a well-spaced configuration in $\Omega_{i}$ with one additional particle added near to an existing particle. We need:

\begin{lemma} [Hitting from Well-Spaced Configurations] \label{LemmaHittingWellMixed}
Fix $m \in \mathbb{N}$ and $1 \leq i \leq m$. Let $\{S_{t}\}_{t \in \mathbb{N}}$ be a simple exclusion process started at a configuration $S_{1} \in \mathcal{G}(i)$ and let 
\be \label{EqDefCollisionTime}
\tau_{\mathrm{coll}} = \min \{t > 0 \, : \, \exists \, u,v \in \LL \text{ s.t. } S_{t}[u] = S_{t}[v] = 1 \text{ and } |u-v| = 1\}
\ee
 be the first time that a collision occurs. Then there exists some $\delta = \delta(m) > 0$ so that 
\be 
\P[\tau_{\mathrm{coll}} > \delta \frac{n^{2}}{\log(n)}] > \frac{\delta}{\log(n)}.
\ee 

\end{lemma}

\begin{proof}

Let $S = \{u \, : \, S_{1}[u] = 1\}$ and let $x_{1},x_{2}$ be two elements of $S$ at distance exactly 2. Let $S_{1}^{(1)} = \textbf{1}_{S \backslash \{x_{2}\}}$, let $S_{1}^{(2)} = \textbf{1}_{S \backslash \{x_{1}\}}$ and let $S_{1}^{(3)} = \textbf{1}_{\{x_{1},x_{2}\}}$. We let $\{ S_{t}^{(1)} \}_{t \in \mathbb{N}}$, $\{ S_{t}^{(2)} \}_{t \in \mathbb{N}}$ and $\{ S_{t}^{(3)} \}_{t \in \mathbb{N}}$ be simple exclusion processes with these three starting points, coupled to $\{S_{t}\}_{t \in \mathbb{N}}$ by choosing the same update sequence in Definition \ref{DefSimpleExclusion}. Let $\tau_{\mathrm{coll}}^{(1)}$, $\tau_{\mathrm{coll}}^{(2)}$ and $\tau_{\mathrm{coll}}^{(3)}$ be their associated collision times, given by the formula
\be 
\tau_{\mathrm{coll}}^{(\ell)} = \min \{t > 0 \, : \, \exists \, u,v \in \LL \text{ s.t. } S_{t}^{(\ell)}[u] = S_{t}^{(\ell)}[v] = 1 \text{ and } |u-v| = 1 \}
\ee 
for $\ell \in \{1,2,3\}$. We note that, under this coupling of the four simple exclusion processes, any single particle in $S_{t}$ appears in at least one of $S_{t}^{(1)},S_{t}^{(2)},S_{t}^{(3)}$:
\be 
\{ u \, : \, S_{t}[u] = 1 \} = \cup_{\ell=1}^{3} \{ u \, : \, S_{t}^{(\ell)}[u] = 1 \}.
\ee 

Furthermore, any \textit{pair} of particles in $S_{t}$ appears in at least one of $S_{t}^{(1)},S_{t}^{(2)},S_{t}^{(3)}$:
\be 
\{ (u,v) \, : \, S_{t}[u] = S_{t}[v] = 1 \} = \cup_{\ell =1}^{3} \{ (u,v) \, : \, S_{t}^{(\ell)}[u] = S_{t}^{(\ell)}[v] = 1 \}.
\ee 

Since $\tau_{\mathrm{coll}}$ and $\{ \tau_{\mathrm{coll}}^{(\ell)}\}_{\ell=1}^{3}$ are determined by the positions of pairs of particles, this implies 
\be \label{IneqRelatingCollisionTimes_2}
\tau_{\mathrm{coll}} = \min(\tau_{\mathrm{coll}}^{(1)}, \tau_{\mathrm{coll}}^{(2)}, \tau_{\mathrm{coll}}^{(3)}).
\ee 
By Theorem 4 of \cite{Cox89}, there exists some $\delta_{1} = \delta_{1}(m)$ so that 
\be [IneqRelatingCollisionsEasyBit1_2]
\P[\tau_{\mathrm{coll}}^{(1)} <  \delta_{1} \frac{n^{2}}{\log(n)}] &< \frac{1}{\log(n)^{2}} \\
\P[\tau_{\mathrm{coll}}^{(2)} <  \delta_{1} \frac{n^{2}}{\log(n)}] &< \frac{1}{\log(n)^{2}} 
\ee 
uniformly in $1 \leq i \leq m$. By Theorem 4.1 of \cite{jain1968range}, there exists some $0 < \delta_{2}, C < \infty$ so that 
\be  \label{IneqRelatingCollisionsEasyBit2_2}
\P[\tau_{\mathrm{coll}}^{(3)} < \delta_{2} \frac{n^{2}}{\log(n)}] < 1 - \frac{C}{\log(n)}.
\ee 
Combining Inequalities \eqref{IneqRelatingCollisionsEasyBit1_2} and \eqref{IneqRelatingCollisionsEasyBit2_2}, there exists some $\delta = \delta(m)$ and constant $0 < C < \infty$ so that 
\be 
\P[\min(\tau_{\mathrm{coll}}^{1}, \tau_{\mathrm{coll}}^{2}, \tau_{\mathrm{coll}}') < \frac{\delta n^{2}}{\log(n)}] < 1 - \frac{C}{\log(n)}.
\ee 
Combining this with Inequality \eqref{IneqRelatingCollisionTimes_2} completes the proof.
\end{proof}

For fixed $i$, let $\{Z_{t}\}_{t \in \mathbb{N}}$ be a Markov chain with kernel $P_{i}$,  let $\tau^{(i)} = \tau^{(i)}(1) = \min \{ t \in \mathbb{N} \, : \, Z_{t} \in \Omega_{i} \}$ and let $\tau^{(i+1)} = \tau^{(i+1)}(1) = \min \{ t \in \mathbb{N} \, : \, Z_{t} \in \Omega_{i+1} \}$. For $j \in \mathbb{N}$, we define inductively 
\be 
\tau^{(i)}(j+1) &= \min \{ t > \tau^{(i+1)}(j) \, : \, Z_{t} \in \Omega_{i} \} \\
\tau^{(i+1)}(j+1) &= \min \{ t > \tau^{(i)}(j) \, : \, Z_{t} \in \Omega_{i+1} \}. \\
\ee 

We have

\begin{corollary}[Collision from Well-Spaced Configurations] \label{CorrHittingWellMixed}

Fix $m \in \mathbb{N}$ and $1 \leq i \leq m$. Let $\{Z_{t}\}_{t \in \mathbb{N}}$ be as above, with initial configuration $Z_{1} \in \mathcal{G}(i)$. Then there exists some $\delta = \delta(c,m) > 0$ so that 
\be 
\P[\tau^{(i)} > \delta \frac{n^{3}}{\log(n)}] > \frac{\delta}{\log(n)}.
\ee 
\end{corollary}

\begin{proof}

We consider a simple exclusion process $\{S_{t}\}_{t \in \mathbb{N}}$ started at $S_{1} = Z_{1}$. We let $\tau_{\mathrm{coll}}$ be as in Equation \eqref{EqDefCollisionTime}. By analyzing the maximal coupling of $S_{t}$ and $Z_{t}$, it is straightforward to check that there exists some $0 < \gamma, \epsilon_{0} < 1$ so that 

\be \label{InMaximalCouplingCollProb}
\P[\tau^{(i)} > \gamma \, \epsilon n^{3}] \geq \gamma \P[\tau_{\mathrm{coll}} > \epsilon n^{2}]
\ee 
for all $0 < \epsilon < \epsilon_{0}$. Applying Lemma \ref{LemmaHittingWellMixed} completes the proof. Note that a detailed proof of Inequality \eqref{InMaximalCouplingCollProb} is given in the first half of the proof of Lemma 7.4 of \cite{pillai2015mixing}.
\end{proof}

We have also have the following bounds on return times:   

\begin{lemma} \label{LemmaBunchOfWeakBounds1}
Fix $0 < r < \infty$. There exists some $C_{1} = C_{1}(c,r)$ so that
\be  \label{IneqWeakFinal1}
\max_{z \in \Omega_{i} \cup \Omega_{i+1}} \E[\tau^{(i)}]  \leq C_{1} \, n^{3} \log(n)
\ee 
uniformly in $1 \leq i \leq r \log(n)$. There exists some $C_{2} = C_{2}(c,r)$ so that

\be \label{IneqWeakFinal2}
\max_{z \in \Omega_{i} \cup \Omega_{i+1}} \E[\tau^{(i+1)}]  \leq C_{2} \, n^{3} 
\ee 

uniformly in $1 \leq i \leq r \log(n)$. There exists some $C_{3} = C_{3}(c,r)$ so that 
\be \label{IneqWeakFinal3}
\min_{z \in \Omega_{i}} \P[\tau^{(i+1)} > \frac{n^{3}}{C_{3} \log(n)^{2}}] \geq C_{3}^{-1}.
\ee 
uniformly in $1 \leq i \leq r \log(n)$.

\end{lemma}

\begin{proof}

To prove Inequality \eqref{IneqWeakFinal2}, note  that we can simulate a step of the Markov chain $\{Z_{t}\}_{t \in \mathbb{N}}$ in terms of the KCIP $\{X_{t}\}_{t \in \mathbb{N}}$ with starting point $X_{1} = Z_{1}$ according to the following rather inefficient rejection-sampling algorithm:

\begin{defn} [Coupling of Trace and KCIP] \label{DefRestAlg}
With notation as above, the following is a valid coupling of the KCIP and one step of its trace:
\begin{enumerate}
\item Simulate $\{X_{t}\}_{t \in \mathbb{N}}$. 
\item Define $\eta = \min \{ s > 1 \, : \, X_{s} \in \cup_{j} \Omega_{j}\}$.
\item If $X_{\eta} \in \Omega_{i} \cup \Omega_{i+1}$, set $Z_{2} = X_{\eta}$. Otherwise, go back to step 1.
\end{enumerate}
\end{defn}

We now analyze this algorithm.  Recall $G_t$ and $\mathrm{ConnComp}(G_t)$ introduced in the beginning of Section \ref{SecDriftCond}. Fix $X_{1} \in \Omega_{i}$ and define the events and random times 
\be 
\mathcal{A} &= \{ \exists \, v \in \LL \, : \,X_{2} = X_{1} \cup \{v\} \text{ and } |\{u \, : \, X_{1}[u] = 1, \, |u-v| \leq 1 \}| = 1 \} \\
\kappa &= \inf \{s > 2 \, : \, X_{s} \neq X_{s-1} \} \\
\mathcal{B} &= \{ |X_{\kappa}| > |X_{2}|, \, \mathrm{ConnComp}(G_{\kappa}) = \mathrm{ConnComp}(G_{1}) \} \cap \mathcal{A} \\
\zeta &= \inf \{s > \kappa \, : \, X_{s} \neq X_{s-1} \} \\
\mathcal{C} &= \{X_{\zeta} \in \Omega_{i+1} \} \cap \mathcal{A} \cap \mathcal{B}.
\ee 
By direct computation,
\be 
\P[\mathcal{A}] &\geq \frac{c}{n^{2}} \\
\E[\textbf{1}_{\mathcal{B}} | \mathcal{A}] &\geq \textbf{1}_{\mathcal{A}}(\frac{c}{2 n}  - O(\frac{r \log(n)}{n^{2}}))\\
\E[\textbf{1}_{\mathcal{C}} | \mathcal{A}, \mathcal{ B}] &\geq \textbf{1}_{\mathcal{A} \cap \mathcal{B}}(\frac{1}{4} - O(\frac{r \log(n)}{n})).
\ee 
Combining these bounds, we have 
\be 
\P[Z_{1} \in \Omega_{i+1}] \geq \P[ \mathcal{C}] \geq \frac{C}{n^{3}}
\ee 
for some $C = C(r,c) > 0$. This completes the proof of Inequality \eqref{IneqWeakFinal2}.

Inequality \eqref{IneqWeakFinal1} is proved exactly as the first inequality in Lemma 7.6 of \cite{pillai2015mixing}, with one small change: the single reference to  Theorem 5 of \cite{Cox89} should be replaced by a reference to Theorem 1.1 of \cite{Oliv12b}. Inequality \eqref{IneqWeakFinal3} is exactly Lemma 4.1 of \cite{pillai2015mixing}.
\end{proof}

We also have:

\begin{lemma}  \label{LemmaBunchOfWeakBounds2}
For fixed $m \in \mathbb{N}$,  there exist constants $C_{1} = C_{1}(m,c), C_{2} = C_{2}(m,c)$, $C_{3} = C_{3}(m,c)$  so that
\be 
\min_{z \in \Omega_{i+1}} \P[\tau^{(i)}(C_{2} \log(n)^{2}) > \frac{n^{3}}{C_{1} \, \log(n)^{3}}] \geq \frac{C_{3}}{\log(n)}
\ee 
uniformly in $1 \leq i \leq m$. 
\end{lemma}

\begin{proof}

Let $\{Z_{t}\}_{t \in \mathbb{N}}$ be a Markov chain evolving according to $P_{i}$, with $Z_{1} \sim \mathrm{unif}(\Omega_{i})$. Define the measure $\mu_{i}$ on $\Omega_{i+1}$ by
\be 
\mu_{i}(A) = \P[Z_{1} \in A | Z_{1} \in \Omega_{i+1}].
\ee
Recall the definition of $\mathcal{G}(i)$ in Definition \ref{DefWellSpaced}. We note that, by the usual `coupon collector' problem and the observation that $1 - o(1)$ of the transitions in $P_{i}$ correspond to adding a single particle in the underlying KCIP (see Definition \ref{DefRestAlg} for a precise coupling of $P_{i}$ and the KCIP which makes this fact clear), we have 
\be 
\mu_{i}(\mathcal{G}(i)) = 1 - o(1).
\ee

Combining the hitting and occupation bounds of Lemma \ref{LemmaBunchOfWeakBounds1} with the bound on the mixing time $\tmix^{(i)}$ given in Lemma \ref{CorMixingTimeRestWalk}, for all $\epsilon > 0$ there exists a constant $L = L(\epsilon,m,c)$ so that 
\be \label{IneqHittingTheGoodSet}
\P[Y_{\tau^{(i+1)}(L \log(n)^{2})} \in \mathcal{G}(i)] \geq \mu_{i}(\mathcal{G}(i)) - \epsilon = 1 - \epsilon - o(1).
\ee 

By Corollary \ref{CorrHittingWellMixed}, there exists some $\delta = \delta(m,c)$ so that
\be 
\P[\tau^{(i)}(L \log(n)^{2} + 1) - \tau^{(i)}(L \log(n)^{2}) > \delta \frac{n^{3}}{\log(n)^{2}} | Y_{\tau^{(i+1)}(L \log(n)^{2})} \in \mathcal{G}(i)]  \geq \frac{\delta}{\log(n)}.
\ee 
Combining this with Inequality \eqref{IneqHittingTheGoodSet} completes the proof.

\end{proof}

\begin{lemma} \label{LemmaIneqRealRestrictionDoubleGap}
Fix $0 < r < \infty$. Then there exists $C = C(r,c)$ so that 
\be 
1 - \beta_{1}(P_{i}) \geq \frac{C}{ n^{3} \log(n)^{9}}
\ee 
uniformly in $1 \leq i \leq r \log(n)$.
\end{lemma}

\begin{proof}

For $T \in \mathbb{N}$, let $N_{i}(T) = |\{ 0 \leq t \leq T \, : \, Z_{t} \in \Omega_{i} \}|$ and $N_{i+1}(T) = |\{ 0 \leq t \leq T \, : \, Z_{t} \in \Omega_{i+1} \}|$. By Lemmas \ref{LemmaBunchOfWeakBounds1} and \ref{LemmaBunchOfWeakBounds2}, for all $M \in \mathbb{N}$ there exists some $C = C(r,c,M)$ so that 
\be \label{IneqFinalOccupationBoundReally1}
\P[N_{i}(T) < M n^{3} \log(n)^{3}] \leq \frac{1}{100}
\ee 
for all $T > C \, n^{3} \log(n)^{5}$, uniformly in $1 \leq i \leq r \log(n)$. The same lemmas imply that for all $m, M \in \mathbb{N}$, 
there exists some $C = C(m,c,M)$ so that 
\be \label{IneqFinalOccupationBoundReally2}
\P[N_{i+1}(T) < M n^{3} \log(n)^{3}] \leq \frac{1}{100}
\ee 
for all $T > C \, n^{3} \log(n)^{9}$, uniformly in $1 \leq i \leq m$.

Let $m_{\max} = 100 \, \max(1,c)$. We note that, for $i \geq m_{\max}$ and all $n > N_{0}$ sufficiently large, $\frac{\pi(\Omega_{i})}{\pi(\Omega_{i} \cup \Omega_{i+1})} > 0.51$. 

Combining the occupation bound in Inequalities \eqref{IneqFinalOccupationBoundReally1} and \eqref{IneqFinalOccupationBoundReally2} with the mixing bound in Lemma \ref{CorMixingTimeRestWalk}, for all $ 0 < C < \infty$ there exists $M = M(C,c)$
\be 
\P[N_{i}(T) > C \tau_{n,i}] &> 0.99, \qquad 100 m_{\max} <i \leq 2 \log(n)  \\
\P[ \{N_{i}(T) > C \tau_{n,i}\} \cap \{N_{i+1}(T) > C \tau_{n,i+1}\}] &> 0.99, \qquad 1 \leq i \leq 100 m_{\max} \\
\ee 
for all $T > M n^{3} \log(n)^{9}$. The result now immediately follows from Lemma 2.1 of \cite{PiSm15} and the observation that
\be 
\frac{\pi(\Omega_{i})}{\pi(\Omega_{i} \cup \Omega_{i+1})} > 0.51
\ee 
for all $i > 100 m_{\max}.$

\end{proof}

\subsection{Proof of Lemma \ref{LemmaMixingModDensity}}

Applying Theorem \ref{ThmQuoteRandallMadras}, with bounds on the individual spectral gaps given by Lemmas \ref{LemmaIneqTildeGap} and \ref{LemmaIneqRealRestrictionDoubleGap}, there exists some $0 < C = C(r,c) < \infty$ so that
\be 
1 - \beta_{1}(P_{n, r \log(n)}) \geq \frac{C}{n^{3} \log(n)^{11}}.
\ee 

Applying the standard bound on the mixing time of a finite Markov chain in terms of its spectral gap (see \textit{e.g.} Theorem 12.3 of \cite{LPW09}) completes the proof of Lemma \ref{LemmaMixingModDensity}. 

\section{Proof of Theorem \ref{ThmMainResult}} \label{SecProofThm}
 Theorem 3 of \cite{pillai2015mixing} yields that, there exists some constant $C = c(c)$ so that 
\be 
\tmix \geq C \, n^{3}.
\ee 
We now prove the upper bound on $\tmix$. For fixed $0 < r < \infty$ and $T \in \mathbb{N}$, define the occupation time
\be 
N(r,T) = | \{ 1 \leq t \leq T \, : \, X_{t} \in \cup_{i=1}^{\lfloor r \log(n) \rfloor } \Omega_{i} \} |.
\ee 

We claim:
\begin{prop} \label{LemmaOccMeasureBound}
With notation as above, there exists some $r = r_{\max}(c) < \infty$ and $C_{1} = C_{1}(c,r)$, $C_{2} = C_{2}(c,r)$ so that
\be \label{IneqOccupationReallyLastOne}
\P[N(r,T) \leq C_{1} n^{3} \log(n)^{13}] \leq \frac{1}{100}
\ee 
for all $T > C_{2} n^{3} \log(n)^{14}$. 
\end{prop}
\begin{proof}
Fix $\epsilon_{0}$  as in the statement of Theorem \ref{LemmaContractionEstimate}, let $\epsilon = \frac{1}{2} \epsilon_{0}$, and let $\alpha$, $C_{G}$ and $\{V_{t}\}_{t \in \mathbb{N}}$ be as in the statement of Theorem \ref{LemmaContractionEstimate}. Fix $r = \frac{2}{\alpha} C_{G}$ and define $\ck = \{x \in \{0,1\}^{\LL} \, : \, \sum_{v \in \LL} x[v] \leq r \, \log(n)\}$. Let $\tau_{\mathrm{start}} = \inf \{ t \in \mathbb{N} \, : \, X_{t} \in \ck \}$ and fix $k \in \mathbb{N}$. By Theorem \ref{LemmaContractionEstimate},
\be 
\E[V_{k \epsilon n^{3} \log(n)} \textbf{1}_{\tau_{\mathrm{start}} > k \epsilon n^{3} \log(n)} ] \leq \big(1 - \frac{1}{2} \alpha \big)^{k} V_{1},
\ee 
and so by Markov's inequality and the trivial bound that $V_{t} \leq n$ for all $t$,
\be \label{IneqLemmaTechLem2Start}
\P[\tau_{\mathrm{start}} > k \epsilon n^{3} \log(n)] &\leq \P[V_{k \epsilon n^{3}} \textbf{1}_{\tau_{\mathrm{start}} > k \epsilon n^{3}} > 1 ]\\
&\leq n \big(1 - \frac{1}{2} \alpha \big)^{k}. 
\ee 
Fix $T \in \mathbb{N}$ and constants $0 < C_{1},C_{2} < \infty$. Let  $\{ Z_{i}\}_{i \in \mathbb{N}}$ be an i.i.d. sequence of random variables with geometric distribution and mean $\frac{2}{\alpha}$. By Inequality \eqref{IneqLemmaTechLem2Start}, the Markov property and Lemma 7.1 of \cite{pillai2015mixing}, 
\be \label{IneqFinalLongCalc}
\P[ & \sum_{t=1}^{ C_{1} n^{3} \log(n)^{14}} \textbf{1}_{X_{t} \in \ck} > C_{2} n^{3} \log(n)^{13}] \geq \P[\sum_{t=1}^{ C_{1} n^{3} \log(n)^{14}} \textbf{1}_{X_{t} \in \ck} > C_{2}  n^{3} \log(n)^{13} | \tau_{\mathrm{start}} < T] \\
&\hspace{8cm} \times \P[\tau_{\mathrm{start}} < T] \\
& = \P[\tau_{\mathrm{start}} < T]  \sum_{s=1}^{T} \P[\sum_{t=s}^{ C_{1} n^{3} \log(n)^{14}} \textbf{1}_{X_{t} \in \ck} > C_{2} n^{3} \log(n)^{13} | \tau_{\mathrm{start}} = s] \P[\tau_{\mathrm{start}} = s | \tau_{\mathrm{start}} \leq T]\\
&\geq \big(1 - n \big(1 - \frac{1}{2} \alpha \big)^{\lfloor\frac{T}{\epsilon n^{3} \log(n)} \rfloor} \big)  \P \big[\sum_{t=T}^{ C_{1} n^{3} \log(n)^{14}} \textbf{1}_{X_{t} \in \ck} > C_{2} n^{3} \log(n)^{13} | \tau_{\mathrm{start}} \leq T] \\
&\geq \big(1 - n \big(1 - \frac{1}{2} \alpha \big)^{\lfloor\frac{T}{\epsilon n^{3} \log(n)} \rfloor} \big)  (1-  \lceil \epsilon n^{3} \log(n) \rceil \P \big[\sum_{j=1}^{C_{2} \log(n)^{13}} Z_{j} \leq C_{1}  \log(n)^{14} - \frac{T}{\epsilon n^{3} \log(n)} \big]).
\ee 
Choosing $T = \lfloor \frac{C_{1}}{2} n^{3} \log(n)^{14} \rfloor$, we have for $C_{1}$ sufficiently large that
\be 
\P \big[\sum_{j=1}^{C_{2} \log(n)^{13}} Z_{j} \leq C_{1}  \log(n)^{14} - \frac{T}{\epsilon n^{3} \log(n)} \big] = o(n^{-10})
\ee  
by a standard concentration inequality for geometric random variables. Combining this with the calculation \eqref{IneqFinalLongCalc}  completes the proof. 

\end{proof}

The upper bound on $\tmix$ now follows immediately from Lemma 2.1 of \cite{PiSm15} as explained in {\bf STEP1} of Section \ref{SecRoadmap}, with the bound on the occupation time given by Inequality \eqref{IneqOccupationReallyLastOne} and the bound on the maximal mixing time given by Theorem \ref{LemmaMixingModDensity}.

\bibliographystyle{plain}
\bibliography{CIBib}

\begin{thebibliography}{10}

\bibitem{AldList}
David Aldous.
\newblock Open problems, 2017.
\newblock [Online at
  \url{https://www.stat.berkeley.edu/~aldous/Research/OP/index.html}; last
  accessed 7-June-2017].

\bibitem{CMRT09}
Nicoletta Cancrini, Fabio Martinelli, Cyril Roberto, and Cristina Toninelli.
\newblock Kinetically constrained models.
\newblock {\em New Trends in Mathematical Physics}, pages 741--752, 2009.

\bibitem{CFM14}
Paul Chleboun, Alessandra Faggionato, and Fabio Martinelli.
\newblock Time scale separation and dynamic heterogeneity in the low
  temperature {E}ast model.
\newblock {\em Communications in Mathematical Physics}, 328(3):955--993, 2014.

\bibitem{CFM15}
Paul Chleboun, Alessandra Faggionato, and Fabio Martinelli.
\newblock Mixing time and local exponential ergodicity of the {E}ast-like
  process in $\mathbb{Z}^{d}$.
\newblock {\em ar{X}iv preprint. ar{X}iv:1501.02240}, 2015.

\bibitem{CFM14b}
Paul Chleboun, Alessandra Faggionato, and Fabio Martinelli.
\newblock Relaxation to equilibrium of generalized {E}ast processes on
  $\mathbb{Z}^{d}$: Renormalization group analysis and energy-entropy
  competition.
\newblock {\em The Annals of Probability}, 44(3):1817--1863, 2016.

\bibitem{ChMa13}
Paul Chleboun and Fabio Martinelli.
\newblock Mixing time bounds for oriented kinetically constrained spin models.
\newblock {\em Electronic Journal of Probability}, 18:1--9, 2013.

\bibitem{ClSu73}
Peter Clifford and Aida Sudbury.
\newblock A model for spatial conflict.
\newblock {\em Biometrika}, 60(3):581--588, 1973.

\bibitem{Cox89}
J.~Theodore Cox.
\newblock Coalescing random walks and voter model consensus times on the torus
  in $\mathbb{Z}^{d}$.
\newblock {\em Annals of Probability}, 17(4):1333--1366, 1989.

\bibitem{DiSa93b}
Persi Diaconis and Laurent Saloff-Coste.
\newblock Comparison theorems for reversible {M}arkov chains.
\newblock {\em Annals of Applied Probability}, 3(3):696--730, 1993.

\bibitem{DiSa96c}
Persi Diaconis and Laurent Saloff-Coste.
\newblock Logarithmic {S}obolev inequalities for finite {M}arkov chains.
\newblock {\em Annals of Applied Probability}, 6:695--750, 1996.

\bibitem{DGJM06}
Martin Dyer, Leslie Goldberg, Mark Jerrum, and Russell Martin.
\newblock {M}arkov chain comparison.
\newblock {\em Probability Surveys}, 3:89--111, 2006.

\bibitem{FrAn84}
G.~Fredrickson and H.~Andersen.
\newblock Kinetic {I}sing model of the glass transition.
\newblock {\em Physical Review Letters}, 53:1244--1247, 1984.

\bibitem{FrAn85}
G.~Fredrickson and H.~Andersen.
\newblock Facilitated kinetic {I}sing models and the glass transition.
\newblock {\em Journal of Chemical Physics}, 83:5822--5831, 1985.

\bibitem{GST11}
Juan Garrahan, Peter Sollich, and Cristina Toninelli.
\newblock {\em Dynamical Heterogeneities in Glasses, Colloids and Granular
  Media}, chapter~10.
\newblock Oxford University Press, 2011.

\bibitem{hamamuki2014discrete}
Nao Hamamuki.
\newblock A discrete isoperimetric inequality on lattices.
\newblock {\em Discrete \& Computational Geometry}, 52(2):221--239, 2014.

\bibitem{HoLi75}
Richard Holley and Thomas Liggett.
\newblock Ergodic theorems for weakly interacting infinite systems and the
  voter model.
\newblock {\em Annals of Probability}, 3(4):643--663, 1975.

\bibitem{jain1968range}
N~Jain and S~Orey.
\newblock On the range of random walk.
\newblock {\em Israel Journal of Mathematics}, 6(4):373--380, 1968.

\bibitem{JSTV04}
Mark Jerrum, Jung-Bae Son, Prasad Tetali, and Eric Vigoda.
\newblock Elementary bounds on {P}oincare and log-{S}obolev constants for
  decomposable {M}arkov chains.
\newblock {\em Annals of Applied Probability}, 14(4):1741--1765, 2004.

\bibitem{KoLa06}
George Kordzakhia and Steven Lalley.
\newblock Ergodicity and mixing properties of the {N}ortheast model.
\newblock {\em Journal of Applied Probability}, 43(3):782--792, 2006.

\bibitem{LPW09}
David Levin, Yuval Peres, and Elizabeth Wilmer.
\newblock {\em {M}arkov Chains and Mixing Times}.
\newblock American Mathematical Society, Providence, Rhode Island, 2009.

\bibitem{MaRa02}
Neil Madras and Dana Randall.
\newblock {M}arkov chain decomposition for convergence rate analysis.
\newblock {\em Annals of Applied Probability}, 12:581--606, 2002.

\bibitem{MaTo13}
Fabio Martinelli and Cristina Toninelli.
\newblock Kinetically constrained spin models on trees.
\newblock {\em Annals of Applied Probability}, 23(5):1967--1987, 2013.

\bibitem{martinelli2016towards}
Fabio Martinelli and Cristina Toninelli.
\newblock Towards a universality picture for the relaxation to equilibrium of
  kinetically constrained models.
\newblock {\em arXiv preprint arXiv:1701.00107}, 2016.

\bibitem{morris2005evolving}
Ben Morris and Yuval Peres.
\newblock Evolving sets, mixing and heat kernel bounds.
\newblock {\em Probability Theory and Related Fields}, 133(2):245--266, 2005.

\bibitem{Oliv12b}
Roberto Oliveira.
\newblock On the coalescence time of reversible random walks.
\newblock {\em Transactions of the American Mathematical Society},
  364(4):2109--2128, 2012.

\bibitem{PaTe96}
Jose Palacios and Prasad Tetali.
\newblock A note on expected hitting times for birth and death chains.
\newblock {\em Statistics and Probability Letters}, 30:119--125, 1996.

\bibitem{PeSo13}
Yuval Peres and Perla Sousi.
\newblock Mixing times are hitting times of large sets.
\newblock {\em Journal of Theoretical Probability}, 9:459--510, 2013.

\bibitem{PiSm15}
Natesh Pillai and Aaron Smith.
\newblock Elementary bounds on mixing times for decomposable {M}arkov chains.
\newblock {\em Stochastic Processes and their Applications}, 127(9):3068--3109,
  2017.

\bibitem{pillai2015mixing}
Natesh~S Pillai and Aaron Smith.
\newblock Mixing times for a constrained {I}sing process on the torus at low
  density.
\newblock {\em The Annals of Probability}, 45(2):1003--1070, 2017.

\bibitem{Smit14a}
Aaron Smith.
\newblock Comparison theory for {M}arkov chains on different state spaces and
  application to random walk on derangements.
\newblock {\em Journal of Theoretical Probability}, 28(4):1406--1430, 2015.

\end{thebibliography}
\end{document}